\documentclass[10pt]{siamonline250211}

\usepackage{import}
\usepackage{geometry}
\usepackage{setspace}
\usepackage{graphicx}
\usepackage{placeins}

\usepackage{cite}

\usepackage{caption}
\usepackage{subcaption}

\usepackage{amsmath}
\usepackage{amssymb}
\usepackage{amsfonts}
\usepackage{mathtools}
\usepackage{bm,bbm}
\usepackage{float}

\usepackage{algpseudocodex}

\usepackage[textsize=small]{todonotes}
\setuptodonotes{backgroundcolor=lightgray!40!white, linecolor=gray, size=footnotesize}

\usepackage{sectionformat}

\newsiamremark{remark}{Remark}

\graphicspath{{./graphics}}

\title{Avoiding spectral pollution for transfer operators using residuals}
\author{
April~Herwig\footnote[1]{Technical University Munich \email{april.herwig@tum.de}}
\and Matthew~J.~Colbrook\footnote[2]{Cambridge University \email{m.colbrook@damtp.cam.ac.uk}}
\and Oliver~Junge\footnotemark[1]{}
\and Péter~Koltai\footnote[3]{Universität Bayreuth}
\and Julia~Slipantschuk\footnotemark[3]{}
}
\date{}

\begin{document}

\maketitle

\begin{abstract}
Koopman operator theory enables linear analysis of nonlinear dynamical systems by lifting their evolution to infinite-dimensional function spaces. However, finite-dimensional approximations of Koopman and transfer (Frobenius--Perron) operators are prone to spectral pollution, introducing spurious eigenvalues that can compromise spectral computations. While recent advances have yielded provably convergent methods for Koopman operators, analogous tools for general transfer operators remain limited. In this paper, we present algorithms for computing spectral properties of transfer operators without spectral pollution, including extensions to the Hardy--Hilbert space. Case studies—ranging from families of Blaschke maps with known spectrum to a molecular dynamics model of protein folding—demonstrate the accuracy and flexibility of our approach. Notably, we demonstrate that spectral features can arise even when the corresponding eigenfunctions lie outside the chosen space, highlighting the functional-analytic subtleties in defining the ``true'' Koopman spectrum. Our methods offer robust tools for spectral estimation across a broad range of applications. 
\end{abstract}


\section{Introduction}

We consider a dynamical system with state $x$ in a state space $\Omega \subset
\bbR^d$ that evolves according to
$$\setlength\abovedisplayskip{6pt}\setlength\belowdisplayskip{6pt}
    x_{n + 1} = S (x_n),\quad n = 0, 1, 2, \ldots,
$$
where $S : \Omega \to \Omega$ is a (typically nonlinear) map, and $n$ indexes discrete time. Koopman \cite{Koopman31,KoopmanOG} introduced a linear operator that evolves complex-valued functions—called observables—forward in time:
$$\setlength\abovedisplayskip{6pt}\setlength\belowdisplayskip{6pt}
    [\scrK g](x) = [g \circ S](x)=g(S(x)),\quad g \in L^2 (\Omega),\ x \in \Omega. 
$$
This operator acts by composition with the map of the dynamical system, providing a linear ``lifting'' of the nonlinear dynamics. Instead of analyzing the nonlinear evolution of the state $x$, one can study the linear evolution of observables $g\in L^2$ through $g \mapsto g \circ S$. The adjoint $\scrL=\scrK^*$ is known as the \emph{transfer} or \emph{Frobenius--Perron} operator. Koopman theory typically centers on the analysis of the spectral properties of $\scrK$ or $\scrL$, such as eigenvalues, eigenfunctions, spectra, and spectral measures. 

The Koopman operator formalism has experienced a modern resurgence as a key tool in data-driven dynamical systems analysis \cite{mezic2005spectral,Budi_2012,koopmanreview,DMDmultiverse}. In this context, one assumes that the map $S$ is unknown and seeks to study $\scrK$ or $\scrL$ using trajectory data from the system. In applied settings, the Koopman operator has been utilized in fluid mechanics~\cite{fluid1,fluid2}, oceanography~\cite{ocean1,ocean2}, molecular dynamics~\cite{molecule1,molecule2,schutte2013metastability}, among other fields.

For example, suppose an observable $g\in L^2$ (which may represent a sensor state such as voltage, fluid velocity, or simply the full state $g(x) = x$) can be 
written as $g = \sum_j c_j \psi_j$, where the $\psi_j$ are eigenfunctions associated with eigenvalues $\lambda_j$ of $\scrK$. (In general, one must also account for contributions from the continuous spectrum.) Then $g$ evolves as 
$$\setlength\abovedisplayskip{6pt}\setlength\belowdisplayskip{6pt}
    g (x_{n}) = \sum_{j} 
    \lambda_j^n \, c_j \psi_j (x_0) . 
$$
By approximating this decomposition, the Koopman operator framework can provide a simplified model capturing, for example, the macroscopic dynamics of the system.

Spectral decompositions of $\scrK$ and $\scrL$ can reveal valuable dynamical information. For example, the sign structure of eigenfunctions determines almost-invariant or metastable sets \cite{attr,schutte2013metastability}; level sets of eigenfunctions encode isostables for fixed points \cite{MAUROY2013,mauroy2016}; and eigenfunctions can reveal invariant manifolds \cite{mezicapplications} and ergodic partitions \cite{Budi_2012,mezicergodicpartition}. Finite approximate eigendecompositions are used to identify unknown dynamics \cite{KLUS2013} and to perform model reduction~\cite{bieker2020,koopmanreview}. Moreover, the dominant singular value spectrum of the Koopman and Frobenius--Perron operators encodes prevalent non-stationary dynamical features, such as coherent sets~\cite{froyland2010transport,froyland2013analytic,KWNS18,wu2020variational}.

\paragraph{Dynamic Mode Decomposition}
A central challenge is that the Koopman operator is infinite-dimensional, making its spectra and modes analytically intractable except in simple cases. This has motivated a wide range of data-driven algorithms for approximating Koopman spectral quantities from simulation or measurement data.

A widely used method is \emph{Dynamic Mode Decomposition} (DMD), originally developed in fluid mechanics \cite{SchSe08,schmid2010dynamic} and later recognized as approximating the Koopman operator's eigenvalues and modes \cite{rowley2009spectral}. Given time-series data, DMD seeks a best-fit linear evolution of the snapshots, producing modes and associated eigenvalues (reflecting growth or decay rates) that approximate the system's dynamics. The basic DMD framework has since evolved into numerous variants~\cite{jovanovic2014sparsity,edmd,H_Tu_2014,KlKoSch16,colbrook2021rigorousKoop,colbrook2025rigged,boulle2024multiplicative,resdmd,kerneledmd,mpedmd,Williams_2015,Giannakis_2024,valva2024}; see~\cite{DMDmultiverse} for a recent review. An alternative data-driven method for approximating $\scrL$ was proposed in~\cite{entropic}, based on ideas from optimal transport and yielding a kernel-type scheme.

Nevertheless, nearly all DMD variants suffer from three primary shortcomings:
\begin{itemize}
    \item\textit{Large data requirements:} The Koopman operator is typically approximated from finite trajectory data, requiring extensive sampling to capture the dynamics across the state space, which may also suffer from the curse of dimensionality.
    \item\textit{Choice of observables:} The set of observables (the ``dictionary'') used to discretize the Koopman operator should ensure sufficient expressivity. However, selecting a suitable dictionary is often challenging, especially in high-dimensional settings~\cite{Li2017}. To address this, machine learning methods for constructing observables have been proposed~\cite{autoencodernetworks,neuralnetedmd,Bollt_2021,Alford_Lago_2022}, and manifold learning has also been explored~\cite{BeGiHa15,KoWe20}.
    \item\textit{Spectral pollution and spectral invisibility:}
Finite-dimensional approximations of the infinite-dimensional Koopman operator may suffer from \textit{spectral pollution}, where spurious eigenvalues do not correspond to the true spectrum, or \textit{spectral invisibility}, where genuine spectral regions are entirely missed. These issues are especially pronounced when the Koopman spectrum contains continuous components or mixtures of discrete and continuous spectra. For explicit examples and further discussion, see~\cite{DMDmultiverse,colbrook2024limits}, \cite[Section 2.6]{blank2002ruelle}, \cite[Example 2]{mezic2020numerical}, and references therein.
\end{itemize}

Recently, the first two shortcomings have been rigorously proven to be
fundamentally inherent to Koopman spectral analysis, regardless of the algorithm
employed (DMD or otherwise) \cite{colbrook2024limits}. In particular, there are intrinsic barriers to what can be computed, even with unlimited trajectory data. The third shortcoming has been effectively addressed for
$\mathcal{K}$ acting on $L^2$ spaces by the \textit{Residual DMD} (ResDMD)
algorithm\cite{colbrook2021rigorousKoop} (see also \cite{boulle2025convergent} for reproducing kernel Hilbert spaces), which yields provably convergent and practical methods for computing Koopman spectral properties.

In this paper, we explore methods to tackle the third shortcoming for the adjoint
(Frobenius--Perron) operator $\scrL$ on $L^2$, as well as for $\scrK$ acting on spaces other than $L^2$.


\begin{figure}[t]
    \centering
    \begin{subfigure}{0.49\textwidth}
        \centering
        \includegraphics[width=\linewidth,trim={0 0 0 1cm},clip]{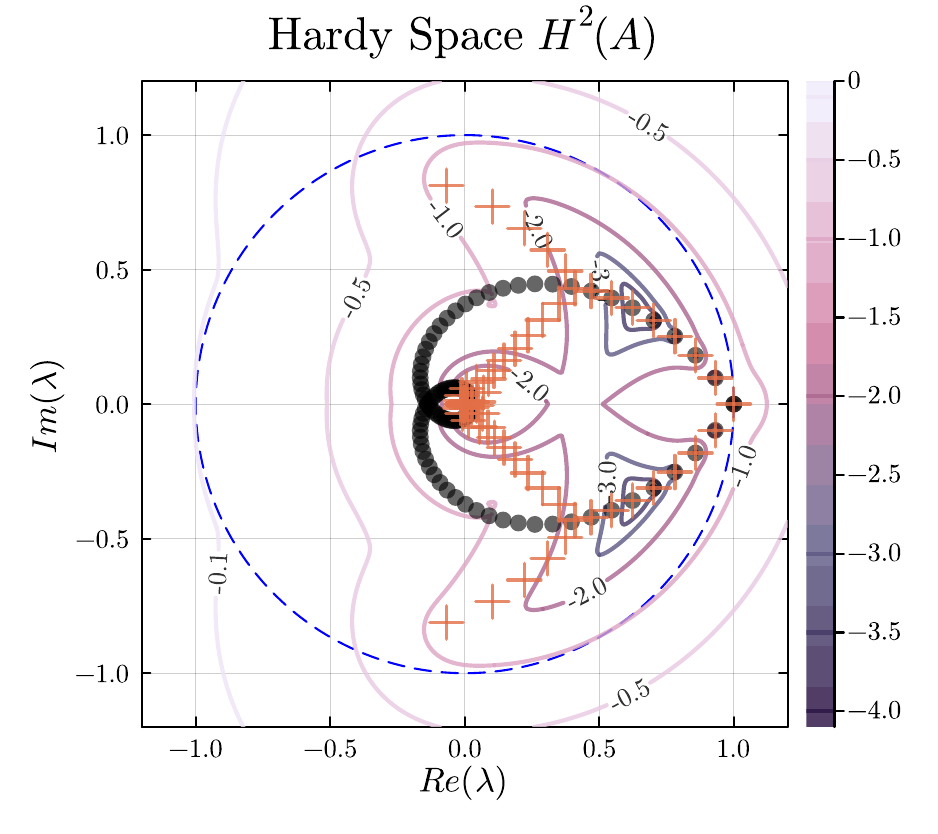} 
    \end{subfigure}
    \hfill
    \begin{subfigure}{0.49\textwidth}
        \centering
        \includegraphics[width=\linewidth,trim={0 0 0 1cm},clip]{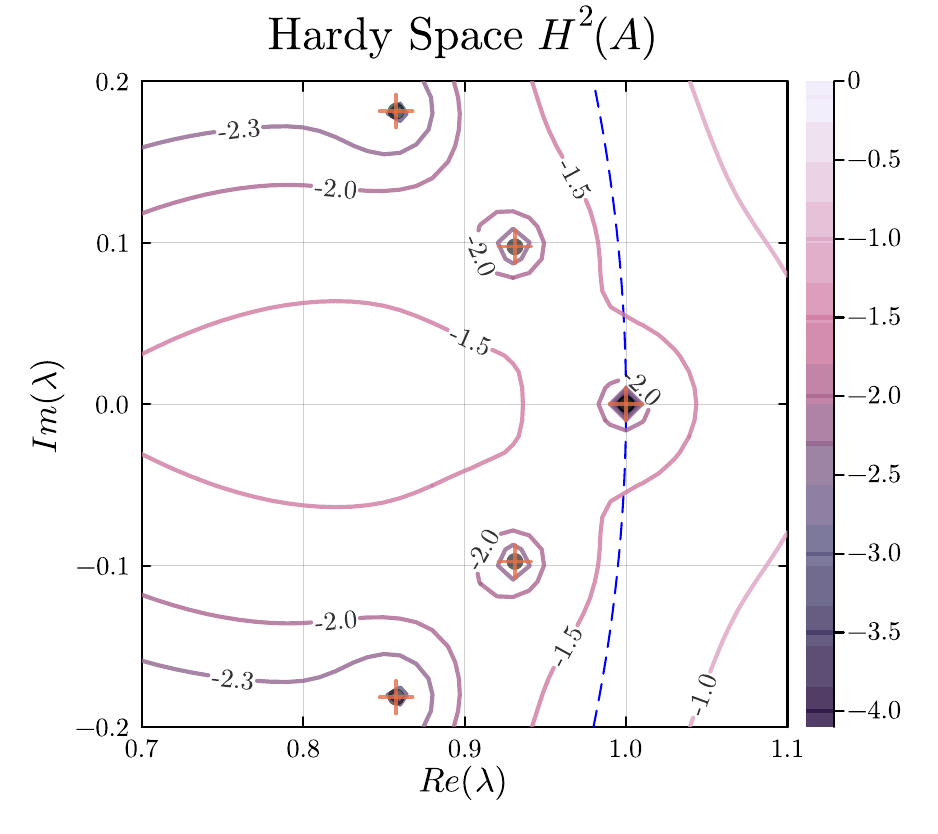} 
    \end{subfigure}
    \caption{Detection of \textit{spectral pollution} using the proposed methods.
The right subfigure is a zoomed-in view of the left for emphasis.
Black dots: true eigenvalues of the Koopman operator for a mixing dynamical system (in a space larger than $L^2$); see \cref{sec:another_blaschke}.
Orange crosses: EDMD approximation of selected leading eigenvalues in the bespoke space. Contour lines approximate boundaries of $\epsilon$-pseudospectra (logarithmic scale), with darker regions indicating higher approximation accuracy and lighter regions indicating possible spectral pollution.}
    \label{fig:blaschke_contour}
\end{figure}

\paragraph{Contributions of this Paper}
First, we develop a method to quantify and control the shortcomings of DMD methods for transfer operators, starting with the widely used kernel EDMD variant~\cite{kerneledmd}. Rather than introducing yet another DMD variant, our approach offers a simple and practical extension compatible with many existing methods, enabling robust error control via the computation of an auxiliary matrix.

Second, we critically examine the subtle interplay between the choice of function space and the spectrum of the Koopman operator—an issue previously highlighted in, e.g.,\cite{dmdanalytic, edmdexpanding, blaschkemaps2024, wormell,froyland2014}. Our analysis focuses on an illustrative example where the Koopman spectrum is known analytically \cite{Slipantschuk}, yet resides in a space strictly larger than $L^2$. We show that numerical approximations of $\scrK$ are influenced by these eigenvalues, even though the corresponding eigenfunctions do not lie in $L^2$. The error control built into our algorithm enables detection and mitigation of this phenomenon, providing a practical diagnostic tool—see \cref{fig:blaschke_contour} for a brief graphical illustration.

We also present computations in function spaces beyond $L^2$, underscoring the broader need for robust numerical methods applicable to both Koopman and Frobenius--Perron operators. This study highlights a key issue of the ``true'' Koopman spectrum: spectral values can influence finite-dimensional approximations even when their eigenfunctions lie outside the chosen observable space. Finally, we validate the proposed methods using real-world datasets, including examples from protein folding dynamics.

\paragraph{Spectral Theory and Notation}
Let $\scrQ: \scrX\to\scrX$ be a bounded linear
operator on a reflexive\footnote{$\scrX^{**}$ can be identified with $\scrX$in
the usual manner. } Banach space of complex-valued functions. The adjoint\footnote{The dual space $\scrX^*$ is the space of
bounded antilinear functionals from $\scrX$ to $\bbC$. } $\scrQ^* : \scrX^* \to
\scrX^*$ satisfies $\scrQ^{**} = (\scrQ^*)^* = \scrQ$.
The spectrum of $\scrQ$ admits a disjoint decomposition
\begin{equation}\setlength\abovedisplayskip{6pt}\setlength\belowdisplayskip{6pt}
    \sigma (\scrQ) = \sigma_p (\scrQ) \uplus \sigma_c (\scrQ) \uplus \sigma_r (\scrQ)
\end{equation}
into point, continuous, and residual spectrum\footnote{
    The spectrum $\sigma
    (\scrQ)$ of $\scrQ$ is the set of complex numbers $\lambda$ for which the
    operator $\scrQ-\lambda I$ does not have a bounded inverse.
    The \emph{point spectrum}
    $\sigma_p (\scrQ)$ contains those $\lambda \in \bbC$ for which $\scrQ - \lambda
    I$ is not injective, so that the kernel of $\scrQ-\lambda I$ is nontrivial,
    i.e., there exists an eigenvector at the eigenvalue $\lambda$. The
    \emph{continuous spectrum} $\sigma_c (\scrQ)$ contains those $\lambda$ for which
    $\scrQ - \lambda I$ is injective, not surjective, but its range is still a dense subset of~$\scrX$. Finally, the \emph{residual spectrum} $\sigma_r (\scrQ)$
    contains those $\lambda$ for which $\scrQ - \lambda I$ is injective, not
    surjective and its range is not dense. 
}. We list some basic facts about the
decomposition of the spectrum~\cite{simonoperatortheory,lax}: $\ (1)$ $\lambda \in \sigma (\scrQ)$ if and only
if $\bar{\lambda} \in \sigma (\scrQ^*)$; $\ (2)$ if $\lambda \in \sigma_r
(\scrQ)$, then $\bar{\lambda} \in \sigma_p (\scrQ^*)$; $\ (3)$ conversely, if
$\lambda \in \sigma_p (\scrQ)$, then $\bar{\lambda} \in \sigma_p (\scrQ^*) \cup
\sigma_r (\scrQ^*)$; $\ (4)$ $\sigma_c (\scrQ) = \sigma_c (\scrQ^*)$. We will
always use $\langle \cdot, \cdot \rangle$ to denote the $L^2$ inner product, and
$\left\langle \cdot \mid \cdot \right\rangle$ for the duality pairing between
$\scrX$ and $\scrX^*$. To ensure clarity, norms will always have a subscript
denoting the space they are referencing, while $\| \cdot \|_{op}$ denotes the
operator norm.

\section{Data-driven Approximations of Operators}

\subsection{Petrov--Galerkin Methods}

We begin with the general Petrov--Galerkin framework for approximating operators~\cite{petrovgalerkin,petrovgalerkinbook}. Let
$$\setlength\abovedisplayskip{6pt}\setlength\belowdisplayskip{6pt}
\overline{ \scrX } = \spn \left\{ \psi_1, \ldots, \psi_N \right\} \subset \scrX, \quad 
\overline{ \scrX^* } = \spn \left\{ \phi_1, \ldots, \phi_M \right\} \subset \scrX^* 
$$
be finite-dimensional trial and test spaces. Define quasi-matrices
$$\setlength\abovedisplayskip{6pt}\setlength\belowdisplayskip{6pt}
\Psi : \bbC^N \to \scrX,\quad
c \mapsto \Psi c = \sum_{j = 1}^{N} c_j \psi_j,
\qquad 
\Phi : \bbC^M \to \scrX^*\quad\text{(analogously).}
$$
We solve the minimization problem
\begin{equation}\setlength\abovedisplayskip{6pt}\setlength\belowdisplayskip{6pt}
    \label{eq:galerkin}
    \left\| Q U - R \right\|_F^2 = \min_{U \in \bbC^{N \times N}} !, \quad 
    Q_{i j} = \left\langle \varphi_i \mid \psi_j \right\rangle , \quad
    R_{i j} = \left\langle \varphi_i \mid \scrQ \psi_j \right\rangle , 
\end{equation}
where $\left\| \cdot \right\|_F$ denotes the matrix Frobenius norm. Here,
$Q$ is the mass (or Gram) matrix, and $R$ is the stiffness 
matrix. When $\scrX$ is a function space, $\Psi$ may also be viewed as a row-vector-valued function
$$\setlength\abovedisplayskip{6pt}\setlength\belowdisplayskip{6pt}
\Omega \ni x \mapsto \Psi (x) = [\, \psi_1 (x) \mid \ldots \mid \psi_N (x) \,],
$$
and similarly for $\Phi$. We may also consider the infinite-dimensional case $N = \infty$, where $\Psi$ and $\Phi$ act on $\ell^2$. A solution to \cref{eq:galerkin} is given by the Moore--Penrose pseudoinverse: $U = Q^\dagger R = ( Q^* Q )^{-1} Q^* R$.

\subsection{Extended Dynamic Mode Decomposition (EDMD)}\label{sec:edmd}

We now apply the Petrov--Galerkin scheme to approximate the Koopman operator and present two perspectives on EDMD~\cite{edmd}: the original formulation and the infinite-data limit as a Galerkin problem.

\paragraph{1. The Petrov--Galerkin Ansatz}
Let $\{\, x_i,\, S (x_i) \, \}_{i = 1}^M$ be a set of ``snapshots'' of the
dynamical system. In the EDMD context, the basis $\Psi$ of $\scrX$ is referred
to as the ``dictionary''. For $\Phi$, take delta distributions $\varphi_i =
\delta_{x_i}$, so that  $\left\langle \varphi_i \mid \psi_j \right\rangle =
\left\langle \delta_{x_i} \mid \psi_j \right\rangle = \psi_j (x_i)$. Then solve the problem \cref{eq:galerkin} for $y_j = \scrK \psi_j$: 
\begin{equation}\setlength\abovedisplayskip{6pt}\setlength\belowdisplayskip{6pt}
    \label{eq:edmd}
    \left\| \YX K - \YY \right\|_F^2 = \min_{K \in \bbC^{N \times N}} !\,,
\end{equation}
where $(\YX)_{i j} = \left\langle \delta_{x_i} \mid \psi_j \right\rangle =
\psi_j (x_i)$ and $(\YY)_{i j} = \left\langle \delta_{x_i} \mid \scrK \psi_j
\right\rangle = \psi_j (\, S (x_i) \,)$. This results in the \emph{EDMD matrix}
$K = \YX^\dagger \YY = \left( \YX^* \YX \right)^{-1} \YX^* \YY$, which approximates
$\scrK$~\cite{KlKoSch16}.

\paragraph{2. The Infinite Data Limit}
Provided the snapshots are suitably sampled, define $G = \frac{1}{M} \YX^*
\YX$ and $A = \frac{1}{M} \YX^* \YY$. Then
$$\setlength\abovedisplayskip{6pt}\setlength\belowdisplayskip{6pt}
    G_{j l} = \frac{1}{M} \sum_{i = 1}^{M} \overline{\psi_j (x_i)} \psi_l (x_i) 
    \xrightarrow{\makebox[1.2cm]{\scriptsize{$M \to \infty$}}} 
    \left\langle \psi_j, \psi_l \right\rangle , \quad j, l = 1, \ldots, N 
$$
and analogously $\lim_{M \to \infty} A_{j l} = \left\langle \psi_j, \scrK \psi_l
\right\rangle$. This convergence can be formalized using quadrature theory; see~\cite{KlKoSch16,colbrook2021rigorousKoop} and~\cite[Section 4.1.3]{DMDmultiverse}. For example, if the data points $x_i$ are sampled randomly from the state space $\Omega$, then $\frac{1}{M} (
\YX^* \YX )_{j l}$ serves as a Monte Carlo quadrature for $\langle \psi_j, \psi_l
\rangle$ and $\smash{ \frac{1}{M} ( \YX^* \YY )_{j l} }$ for $\left\langle \psi_j, \scrK \psi_l \right\rangle$. 

This offers an alternative route to derive~\cref{eq:edmd}. Taking $\scrX = L^2$
for both the trial and test spaces and $\Psi = \Phi$, the Galerkin problem becomes
\begin{equation}\label{eq:edmd_infinite_data}\setlength\abovedisplayskip{6pt}\setlength\belowdisplayskip{6pt}
    \left\| \bmG K - \bmA \right\|_F^2 = \min_{K \in \bbC^{N \times N}} ! ,\quad
		\bmG_{j l} = \left\langle \psi_j,\, \psi_l \right\rangle ,
		\quad
		\bmA_{j l} = \left\langle \psi_j,\, \scrK \psi_l \right\rangle .
\end{equation}
The matrix $\bmG$ is invertible since by assumption $\left\{ \psi_1, \ldots,
\psi_N \right\}$ are linearly independent.
Since rescaling by
$\frac{1}{M}$ does not change the minimizer, we see that $\smash{K = \YX^\dagger
\YY}$ serves as a quadrature approximation of the infinite-data solution
$K = \bmG^{-1} \bmA$.

It is important to note that the second viewpoint is \emph{ambivalent} to the inner product. As long as
one has an approximation scheme for $\left\langle \psi_i, \psi_j
\right\rangle_\scrX$ with respect to some inner product space $\scrX$, then one
can approximate the mass matrix $\bmG$ and the stiffness matrix $\bmA$ and 
therefore understand \cref{eq:edmd_infinite_data} on the 
arbitrary space $\scrX$. In this way, one bypasses the Petrov--Galerkin Ansatz 
entirely.

\paragraph{EDMD for the Frobenius--Perron Operator} 
Noting the form \cref{eq:edmd_infinite_data} of $\bmA$, we have that 
$$\setlength\abovedisplayskip{6pt}\setlength\belowdisplayskip{6pt}
    \overline{\bmA_{l j}}
    = \overline{ \left\langle \psi_l, \scrK \psi_j \right\rangle } 
    = \left\langle \scrK \psi_j, \psi_l \right\rangle 
    = \left\langle \psi_j, \scrL \psi_l \right\rangle.
$$
We will make significant use of $\scrL$, so it is useful to note that the above
calculation yields an equivalent Galerkin method for $\scrL$, namely  
$
    L = (\Psi^* \Psi)^{-1} \Psi^* \,\scrL \Psi = \bmG^{-1} \bmA^*, 
$
or the analogue in terms of finite-data matrices: $L = G^{-1} A^*$;
see~\cite{KlKoSch16}.

\subsection{Kernelized EDMD (kEDMD)}\label{sec:kedmd}

As shown later in \cref{sec:blaschke}, an under-ex\-press\-ive dictionary can lead to catastrophic spectral errors. However, the standard formulation $K = (
\YX^* \YX )^{-1} ( \YY^* \YX )$ scales poorly with the dimension of the system's state space. For instance, representing all multivariate polynomials up to degree 6 in $\bbR^{20}$ requires $N \approx 75,000$. The challenge, then, is to increase $N$ efficiently without making the EDMD matrix prohibitively expensive to compute.

The kernel trick~\cite{kernel_OG} is a widely used technique in machine learning~\cite{campbell2001,hofmann2008,muller2018}. Kernelized EDMD~\cite{kerneledmd}, shown in \cref{alg:kedmd}, uses this idea to make EDMD practical when $N$ is large or even infinite. It relies on a (reduced) singular value decomposition (SVD)
$$\setlength\abovedisplayskip{6pt}\setlength\belowdisplayskip{6pt}
    \tfrac{1}{\sqrt{M}} \YX = Q \Sigma Z^*, \quad 
    Q \in \bbC^{M \times r},\ \Sigma \in \bbC^{r \times r},\ Z \in \bbC^{N \times r}
$$
where $r = \mathrm{rank} (\YX)$, $\Sigma$ is a positive diagonal matrix
and $Q$, $Z$ are isometries. 
In practice, a reduced rank $r \le \mathrm{rank} (\YX)$ is chosen in advance for compression. This SVD is used in~\cite{kerneledmd} to define the \emph{kernel EDMD matrix}
\begin{equation}\setlength\abovedisplayskip{6pt}\setlength\belowdisplayskip{6pt}
    \label{eq:K_hat}
    \widehat{ K } = Z^* K Z = 
    \frac{1}{M} \left( \Sigma^\dagger Q^* \right)
    \YY \YX^*
    \left( Q \Sigma^\dagger \right)\ \in\ \bbC^{r \times r} . 
\end{equation}
The matrix $Z^*$ removes the kernel of $\YX$, ensuring that each nonzero eigenpair $(\lambda, v) \in \bbC \times
\bbC^{r}$ of $\widehat{ K }$ corresponds one-to-one with an eigenpair $(\lambda, Z v) \in \bbC \times \bbC^{M \times M}$ of $K$. One then
reverses the order of multiplication in $G = \frac{1}{M} \YX^* \YX$ and $A =
\frac{1}{M} \YX^* \YY$ to define
$$\setlength\abovedisplayskip{6pt}\setlength\belowdisplayskip{6pt}
    \widehat{ G } = \frac{1}{M} \YX \YX^* , \quad 
    \widehat{ A } = \frac{1}{M} \YY \YX^* . 
$$
Now $Q$ and $\Sigma$ can be obtained via an eigendecomposition since $\widehat{
G } = Q \Sigma^2 Q^*$.

\paragraph{Mercer's Theorem}
We can also state these results in terms of Mercer's theorem \cite[Chapter~3,
Theorem~4]{learning}, \cite{mercerstheorem}: any continuous,
symmetric, positive definite function 
$k : \Omega \times \Omega \to \bbC$ defined on a compact space $\Omega \subset \bbC^d$ 
can be decomposed as 
$$\setlength\abovedisplayskip{6pt}\setlength\belowdisplayskip{6pt}
    k(w, z) = \sum_{l = 1}^{\infty} \mu_l \phi_l (z) \, \overline{\mu_l \phi_l (w)}
    = \Psi (z) \Psi (w)^* \quad\quad \forall\, z, w \in \Omega
$$
where $\Psi (x) = [\ \mu_1 \phi_1 (x) \mid \mu_2 \phi_2 (x) \mid \ldots\ ]$.
The multipliers $\mu_l \geq 0$ form a decreasing sequence converging to $0$.
Moreover, $\{ \phi_l \}_l$ forms an orthonormal basis of $L^2$, and $\{ \mu_l
\phi_l \}_l$ is an orthonormal basis for a reproducing kernel Hilbert space
(RKHS) \cite{learning,RKHS}. We refer to the $\mu_l \phi_l$'s as Mercer features and say
that $N = \infty$ when infinitely many $\mu_l > 0$.

Typically, $k (w, z)$ can be evaluated in $O (d)$ operations. Now observe that
$\widehat{ G }_{i l} = \frac{1}{M} \Psi (x_i) \Psi (x_l)^*$ and $\widehat{ A
}_{i l} = \frac{1}{M} \Psi (x_i) \Psi ( S(x_l) )^*$. This means that 
$$\setlength\abovedisplayskip{6pt}\setlength\belowdisplayskip{6pt}
    \widehat{ G }_{i l} = \frac{1}{M} k ( x_l, x_i ), \quad 
    \widehat{ A }_{i l} = \frac{1}{M} k ( S (x_l), x_i ) , 
$$
and so $\widehat{ G }$, $\widehat{ A }$ can be computed in $O ( d M^2 )$
operations, a significant improvement to computing $O (N^2)$ quadrature problems
over $\Omega \subset \bbC^d$, each of which may require e.g.\ $O (M^d)$
operations. Note that $\Psi$ in Mercer's theorem is often infinite, and often
only implicitly known --- its existence is guaranteed, but an explicit form is
not needed. 

\begin{algorithm}[t]
    \caption{Kernel EDMD \cite{kerneledmd}}
    \label{alg:kedmd}
    \begin{algorithmic}[1]
        \Require kernel $k : \Omega \times \Omega \to \bbC$, data points 
            $\{ x_i \}_{i = 1}^M$, compression $r \leq M$.
        \State Construct $\widehat{G} = (\, \frac{1}{M} k ( x_l, x_i ) \,)_{i, l= 1}^M$, 
            $\ \widehat{A} = (\, \frac{1}{M} k ( S(x_l), x_i ) \,)_{i, l = 1}^M$ 
        \State Compute an eigendecomposition $\widehat{G} = Q \Sigma^2 Q^*$ 
        \State Let $\widetilde{\Sigma} = \Sigma [1:r, 1:r]$, $\widetilde{Q}
        = Q [:, 1:r]$ ($r$ largest singular values and vectors) \State Construct
        $\widehat{K} = ( \widetilde{\Sigma}^\dagger \widetilde{Q}^* )
        \widehat{A}
        ( \widetilde{Q} \widetilde{\Sigma}^\dagger )$
        \State Compute an eigendecomposition $\widehat{K} V = V \Lambda$
        \State \Return Eigenvalues and eigenvectors $\Lambda$, $\widetilde{Q} \widetilde{\Sigma} V$
    \end{algorithmic}
\end{algorithm}

\subsection{Validation of Koopman Eigenpairs (ResDMD)}

If $\Pi$ is the orthogonal projection of a Hilbert space $\scrX$ onto $\spn \{ \psi_j \}_{j = 1}^N$ and $\spn \{ \psi_j \}_{j = 1}^\infty$ is dense in $\scrX$, then $\Pi \scrK \Pi$ converges to $\scrK$ in the
strong operator topology as the dictionary sizes goes to infinity. However, it is well known that the spectrum can be highly unstable in the Hausdorff topology, even for operators close in norm~\cite{Pseudospectra}—let alone for those converging only pointwise. In particular, the eigenvalues of $K$ (the matrix representation of $\Pi \scrK
\Pi$) may have little relation to the true spectrum of $\scrK$: entire spectral regions may be missed (spectral invisibility), or persistent spurious eigenvalues may appear (spectral pollution).

A common remedy is \emph{stochastic blurring}. Instead of evolving deterministically via $S$, each point $x \in \Omega$ is assigned a \emph{distribution} of image points, producing a stochastic dynamical system. The resulting Markov process has an associated (stochastic) Frobenius--Perron operator which, under suitable conditions on the noise, is Hilbert--Schmidt on $L^2
(\Omega)$. In this setting, finite-dimensional approximations converge in norm to the compact operator, and so do the eigenvalues~\cite{attr}. In contrast, we adopt an approach based on pseudospectra that allows us to quantify eigenvalue error directly.

\paragraph{Pseudospectra} 
The $\epsilon$-pseudospectrum of a bounded operator $\scrQ$ on a Hilbert space
$\scrX$ is the smallest superset of the spectrum that is
robust to perturbations of size $\epsilon$
\cite{Pseudospectra}:
\begin{equation}\label{eq:pseudospectrum}\setlength\abovedisplayskip{6pt}\setlength\belowdisplayskip{6pt}
    \begin{aligned}
    \sigma_\epsilon (\scrQ) &= \!\!\bigcup_{\| \scrE \|^{}_{op} \leq \epsilon }\!\!\sigma (\scrQ + \scrE)  \\
    &= \left\{ \smash{\lambda \in \bbC }\,\Big\vert\, 
        \smash{\inf_{\| u \|^{}_{\scrX} = 1} \| (\scrQ - \lambda I) u \|_\scrX \leq \epsilon} 
        \text{ or } \smash{\inf_{\| u \|^{}_\scrX = 1} \| (\scrQ^* - \bar{\lambda} I) u \|_\scrX \leq \epsilon} \right\} .  
    \end{aligned}
\end{equation}
It is immediate that $\sigma_{\epsilon_1} (\scrQ) \subset \sigma_{\epsilon_2}
(\scrQ)$ when $\epsilon_1 \leq \epsilon_2$ and that $\sigma (\scrQ) =
\cap_{\epsilon > 0} \sigma_\epsilon (\scrQ)$.

The necessity to consider the adjoint in the above equation leads to the
so-called $\epsilon$-approximate point pseudospectrum: 
$$\setlength\abovedisplayskip{6pt}\setlength\belowdisplayskip{6pt}
    \sigma_{ap, \epsilon} (\scrQ) = \left\{ 
        \lambda \in \bbC \,\Big\vert\, \inf_{\| u \|^{}_{\scrX} = 1} 
        \left\| (\scrQ - \lambda I) u \right\|_\scrX \leq \epsilon 
    \right\} . 
$$
In many cases $\sigma_{ap, \epsilon} (\scrQ)=\sigma_\epsilon (\scrQ)$, in
particular, whenever $\scrQ$ has no residual spectrum. Notable cases include:
matrices / finite-rank operators, compact operators, normal operators. We define the approximate point spectrum as $\sigma_{ap} (\scrQ) = \cap_{\epsilon > 0} \sigma_{ap, \epsilon} (\scrQ)$.

\paragraph{Application to Koopman Operators} 
From the previous observations, $\lambda \in \sigma_{ap} (\scrK)$ if and only if there exists a sequence $u_N \in \scrX$ with $\| u_N \|^{}_{\scrX} = 1$, such
that for any $\epsilon > 0$,  there exists $N$ sufficiently large with 
\begin{equation}\setlength\abovedisplayskip{6pt}\setlength\belowdisplayskip{6pt}
    \label{eq:K_residual}
    \left\| (\scrK - \lambda I) u_N \right\|_\scrX < \epsilon. 
\end{equation}
Such observables $u_N$ are called pseudoeigenfunctions. Like true eigenfunctions, they capture meaningful dynamical behavior: from \cref{eq:K_residual}, $\| \scrK^n u_N - \lambda^n u_N \|^{}_{\scrX} =
O(n \epsilon)$. The function $u_N$ therefore describes a coherent observable in
state space and $\lambda$ encodes (approximate) decay and oscillation of $u_N$
on finite time spans, for as long as $\lambda^n u_N$ dominates the error~$O(n\epsilon)$. Given an eigenpair $(\lambda, c) \in \bbC \times \bbC^N$ of $K$, which we wish to assess as a candidate eigenpair $(\lambda, \Psi c) \in \bbC \times
\scrX$ of $\scrK$, \cref{eq:K_residual} provides a criterion for doing so.

\paragraph{The Residual Function}\label{sec:residual_function}
From \cref{eq:edmd_infinite_data} we see that the \emph{regression error}
\begin{equation}\label{eq:res}\setlength\abovedisplayskip{6pt}\setlength\belowdisplayskip{6pt}
    \res (\lambda, c; M, N) = \frac{1}{\sqrt{M}} \big\| (\YY - \lambda \YX) c \big\|_{\bbC^M}
\end{equation}
is precisely a quadrature approximation of $\left\| (\scrK - \lambda I) \Psi c
\right\|_{L^2}$. We write $\res (\lambda, c)$ whenever $M$ and $N$ can be inferred from the context. It was shown in \cite{resdmd} that for a candidate
eigenvalue $\lambda$ and associated candidate eigenvector $g = \Psi c$ we have
$\lim_{M \to \infty}^{} \res (\lambda, c; M, N)^2 = \smash{\left\| (\scrK -
\lambda I) g \right\|_{L^2}^2}$. This is because letting $J = \frac{1}{M} \YY^*
\YY$ and expanding \cref{eq:res} yields 
\begin{equation}\setlength\abovedisplayskip{6pt}\setlength\belowdisplayskip{6pt}
    \label{eq:res_expanded}
    \begin{split}
        \quad\quad\res (\lambda, c; M, N)^2 
        \ 
        = \ \, & c^* J c\ 
        -\ \bar{\lambda}\ c^* A c\ 
        -\ \lambda\ c^* A^* c\ 
        +\ | \lambda |^2 c^* G c \  \\ 
        \xrightarrow{\makebox[1.2cm]{\scriptsize{$M \to \infty$}}} 
        \ &
        \left\langle \scrK g, \scrK g \right\rangle 
        - \bar{\lambda} \left\langle g, \scrK g \right\rangle 
        - \lambda \left\langle \scrK g, g \right\rangle
        \ +\ | \lambda |^2 \left\langle g, g \right\rangle =\left\| (\scrK - \lambda I) g \right\|_{L^2}^2 . 
    \end{split}
\end{equation}
Hence, if we consider the minimum of such $g$ over the unit ball in $\overline{
\scrX } = \spn \left\{ \psi_1, \ldots, \psi_N \right\}$,
\begin{equation}\setlength\abovedisplayskip{6pt}\setlength\belowdisplayskip{6pt}
    \label{eq:res_min}
    \res (\lambda; M, N)
        = \min_{c^* G c = 1} \res (\lambda, c; M, N) , 
\end{equation}
then 
$
    \lim^{}_{M \to \infty} \res (\lambda; M, N)
        = \min_{\substack{g \in \spn \overline{\scrX} \\ \| g \|^{}_{L^2} = 1}}
            \ \left\| ( \scrK - \lambda I ) g \right\|_{L^2} .
$
This yields 
$$\setlength\abovedisplayskip{6pt}\setlength\belowdisplayskip{6pt}
    \sigma_{ap, \epsilon} (\scrK) = \left\{ 
        \lambda \in \bbC \,\Big\vert\, 
        \lim_{N \to \infty} \lim_{M \to \infty} \res (\lambda; M, N) \leq \epsilon
    \right\} .  
$$
In particular, if we
calculate some candidate eigenpairs $(\lambda, c)$, compute $\res (\lambda, c;
M, N)$ for some ``sufficiently large'' $M$ and $N$ on each candidate eigenpair
and keep only those which satisfy a threshhold $\res (\lambda, c) < \epsilon$,
then those remaining eigenpairs really are ``close'' to eigenpairs of $\scrK$.
The computation of $\res (\lambda)$ reduces to a generalized eigenvalue problem.
This process is summarized in \cref{alg:resdmd}.

\begin{algorithm}[t]
    \caption{Residual DMD to Compute $\sigma_{ap, \epsilon} (\scrK)$ \cite{resdmd}}
    \label{alg:resdmd}
    \begin{algorithmic}[1]
        \Require Dictionary $\left\{ \psi_j \right\}_{j=1}^N$, 
            points $\left\{ x_i \right\}_{i=1}^M$, 
            grid $\left\{ z_\nu \right\}_{\nu=1}^T \subset \bbC$,
            tolerance $\epsilon$
        \State Construct $\YX = (\, \psi_j (x_i) \,)_{ij}$, 
            $\ \YY = (\, \psi_j ( S(x_i) ) \,)_{ij}$.
        \State Construct $G = \frac{1}{M} \YX^* \YX$, $\ A = \frac{1}{M} \YX^* \YY$, $\ J = \frac{1}{M} \YY^* \YY$
        \For{$z_\nu$}
        \State Compute $U = J - \overline{z_\nu} A - z_\nu A^* + | z_\nu |^2 G$
            \State Compute $\res (z_\nu)=\sqrt{\xi}$, where $\xi$ is the smallest 
             eigenvalue of the eigenproblem $U c = \xi G c$ 
            
        \EndFor
        \State \Return $\left\{ z_\nu \mid \res (z_\nu) < \epsilon \right\}$
            as an approximation for $\sigma_{ap, \epsilon} (\scrK)$
    \end{algorithmic}
\end{algorithm}

\section{A Residual Method for Frobenius--Perron Operators}

We now build a residual method for Frobenius--Perron operators. From this point
forward, we shall always assume (without loss of generality) that $\YX$ has full
rank $\min\{M, N\}$.

\subsection{A Naive First Attempt at Duality}

\cref{alg:resdmd} provides a way to compute the approximate point pseudospectrum
of $\scrK$. To resolve the true pseudospectrum, one needs to compute both,
$\min_{\| g \|^{}_{L^2} = 1} \left\| (\scrK - \lambda I) g \right\|_{L^2}$
\emph{and} $\min_{\| g \|^{}_{L^2} = 1} \left\| (\scrL - \lambda I) g
\right\|_{L^2}$ for $\lambda$'s of interest. 

At first, one might hope to perform the calculations in \cite{resdmd} using
$\scrL$ instead of~$\scrK$. That is, starting with $\| (\scrL - \lambda I)
g\|_{L^2}^2$ and aiming for an expression like \cref{eq:res_expanded}. However, the matrix analogous to $J$ is not computable from the available information.
Specifically, for $g =\Psi c$,
\begin{equation*}\setlength\abovedisplayskip{6pt}\setlength\belowdisplayskip{6pt}
    \begin{split}
    \left\| (\scrL - \lambda I) g \right\|_{L^2}^2
    &= \left\langle (\scrL - \lambda I) g, (\scrL - \lambda I) g \right\rangle \\ 
    &= \left\langle \scrL g, \scrL g \right\rangle 
        \ -\ \bar{\lambda} \left\langle g, \scrL g \right\rangle
        \ -\ \lambda \left\langle \scrL g, g \right\rangle
        \ +\ | \lambda |^2 \left\langle g, g \right\rangle \\ 
    &= \left\langle \scrL g, \scrL g \right\rangle 
        \;\ -\;\ \bar{\lambda}\ c^* \bmA^* c 
        \;\ -\;\ \lambda\ c^* \bmA c 
        \;\ +\;\ | \lambda |^2\ c^* \bmG c,
    \end{split}
\end{equation*}
but the term $\left\langle \scrL g, \scrL g \right\rangle_\scrX$ is not
approximable using $\YX$ and $\YY$. Instead, one could start with a regression
error 
\begin{equation}\setlength\abovedisplayskip{6pt}\setlength\belowdisplayskip{6pt}
    \label{eq:L_resression}
    \left\| \left( A^* - \lambda G \right) c \right\|_{\bbC^N}^2 = \frac{1}{M} \left\|
    \left( \YY^* - \lambda \YX^* \right) \YX c \right\|_{\bbC^N}^2
\end{equation}
similar to \cref{eq:res}, but due to the Galerkin property we know that
$L=G^{-1}A^*$ encodes the action of $\Pi \scrL \Pi$ so that $\lim_{M \to \infty}
\left\| (L - \lambda I) c \right\|_{\bbC^N} = \left\| (\Pi \scrL \Pi - \lambda
I) g \right\|_\scrX$ for $g = \Psi c$. Since $G$ is symmetric and positive definite and $\left\| (A^* - \lambda G) c
\right\|_{\bbC^N} = \left\| G (L - \lambda I) c \right\|_{\bbC^N}$, 
$$\setlength\abovedisplayskip{6pt}\setlength\belowdisplayskip{6pt}
     \min \sigma (G)  
    \left\| (L - \lambda I) c \right\|_{\bbC^N} 
    \leq \left\| (A^* - \lambda G) c \right\|_{\bbC^N} 
    \leq  \max \sigma (G) 
    \left\| (L - \lambda I) c \right\|_{\bbC^N} .
$$
Upon taking the limit $M \to \infty$, this yields
\begin{equation}\setlength\abovedisplayskip{6pt}\setlength\belowdisplayskip{6pt}
    \label{eq:lim_bad_regression}
     \min \sigma (\bmG) 
    \left\| (\Pi \scrL \Pi - \lambda I) g \right\|_{L^2} 
    \leq \lim_{M \to \infty} \left\| (A^* - \lambda G) c \right\|_{\bbC^N} 
    \leq  \max \sigma (\bmG)  
    \left\| (\Pi \scrL \Pi - \lambda I) g \right\|_{L^2} , 
\end{equation}
so that \cref{eq:L_resression} only computes (a scaled version of) the
pseudospectrum of $\Pi \scrL \Pi$. This cannot be used analogously to
\cref{alg:resdmd} to compute $\smash{ \min_{\left\| g \right\|^{}_{L^2} = 1}
\left\| (\scrL - \lambda I) g \right\|_{L^2} }$ since we would need to send $N
\to \infty$ before $M \to \infty$. To derive a method that works, we take a detour through kernelized EDMD.

\subsection{Kernelizing Residual DMD: The Frobenius--Perron Connection}
\label{ssec:kres_FP}

We again seek to identify which candidate eigenvalues from \cref{alg:kedmd} are spurious and which are accurate. Equation~\cref{eq:res_expanded} suggests computing $\| (\scrK - \lambda I) g \|_{L^2}$ using modified features $g = \Psi \, Z v$. However, this fails when $M \leq N$—the regime where kernel methods are most advantageous. As shown in \cite{resdmd}, in this regime we have $\res (\lambda, Z v; M, N) = 0$ for any eigenpair $(\lambda, v)$ of $\widehat{K}$, indicating overfitting of the snapshot data. Thus, we must seek an alternative approach to define a residual for the eigenpair.

Recall from equation \cref{eq:res} that $\res$ has an alternative
representation as a regression error, which we deduced from the regression
problem \cref{eq:edmd}. We could analogously ask if $\widehat{K}$ is also the
solution to some other regression problem. Let 
$$\setlength\abovedisplayskip{6pt}\setlength\belowdisplayskip{6pt}
    \hYX = \tfrac{1}{\sqrt{M}} \YX^* Q \Sigma^\dagger = Z, \quad 
    \hYY = \tfrac{1}{\sqrt{M}} \YY^* Q \Sigma^\dagger . 
$$
Then we have from equation \cref{eq:K_hat} that
$$\setlength\abovedisplayskip{6pt}\setlength\belowdisplayskip{6pt}
    \hYX^\dagger \hYY 
    = Z^* \YY^* Q \Sigma^\dagger
    = \left( \Sigma^\dagger Q^* \right) \frac{1}{M} \YX \YY^* \left( Q \Sigma^\dagger \right)
    = \widehat{K}^* . 
$$
Hence $\widehat{K}^*$ is precisely the solution to the least squares problem 
$$\setlength\abovedisplayskip{6pt}\setlength\belowdisplayskip{6pt}
    \min_{B \in \bbC^{N \times N}} \left\| \hYY - \hYX B \right\|_F . 
$$
This means that for a candidate eigenpair $(\lambda, v)$ of $\widehat{K}^*$, the
regression error is given by 
$$\setlength\abovedisplayskip{6pt}\setlength\belowdisplayskip{6pt}
    \widehat{ \kres } (\lambda, v; M, N) = 
    \left\| \left( \hYY - \lambda \hYX \right) v \right\|_{\bbC^N} . 
$$
As before, we suppress the arguments $M$ and $N$ when appropriate.

\paragraph{Interpretation}
At this point, it is unclear whether $\widehat{ \kres }$ has any physical
meaning. It serves as an error metric for a seemingly arbitrary least squares regression problem involving the matrices $\hYX$ and $\hYY$, which lack a clear interpretation. However, the fact that the \emph{adjoint} $\widehat{K}^*$ solves the least squares problem should raise the suspicion that the residual may be more closely related to the Frobenius--Perron operator than to the Koopman operator.

In \cite{kresdmd}, it is suggested to define 
\begin{equation}\setlength\abovedisplayskip{6pt}\setlength\belowdisplayskip{6pt}
    \label{eq:res_hat_no_v}
    \widehat{ \kres } (\lambda; M, N)
    = \min_{v^* v = 1} \widehat{ \kres } (\lambda, v; M, N) 
\end{equation}
and use this analogously to \cref{alg:resdmd}. This is summarized in
\cref{alg:kresdmd}. Since in the regime $N \leq M$, $\Sigma$ and $Z$ are full
rank (actually $Z$ is unitary), we may make the substitution $v = \Sigma^2 Z^*
c$ in \cref{eq:res_hat_no_v}. Now $Q \Sigma^\dagger \Sigma^2 Z^* = Q \Sigma Z^*
= \tfrac{1}{\sqrt{M}} \YX$ so that 
\begin{equation}\setlength\abovedisplayskip{6pt}\setlength\belowdisplayskip{6pt}
    \label{eq:kres_hat_lambda}
    \widehat{ \kres } (\lambda) 
    = \min_{c^* Z \Sigma^4 Z^* c = 1} \Big\| \frac{1}{\sqrt{M}}
        (\YY^* - \lambda \YX^*) \YX c
     \Big\|_{\bbC^N}
    = \min_{c^* G^2 c = 1} \big\| 
        (A^* - \lambda G) c
     \big\|_{\bbC^N} . 
\end{equation}
From \cref{eq:lim_bad_regression} we know the term $\left\| (A^* - \lambda G) c
\right\|_{\bbC^N}$ in the minimization converges to a scaled version of
$\left\| (\Pi \scrL \Pi - \lambda I) \Psi c \right\|_{L^2}$ as $M \to \infty$.
It follows that in the $M \to \infty$ limit the condition $\smash{c^* G^2 c =
c^* \sqrt{G}^* G \sqrt{G} c = 1}$ becomes $\smash{\| \Psi\,  \sqrt{\bmG} c
\|_{L^2}^2 = 1}$. Hence, with 
\begin{equation}\setlength\abovedisplayskip{6pt}\setlength\belowdisplayskip{6pt}
    \label{eq:GammaDef}
    \Gamma \quad = \quad \min_{\substack{
        c \in \bbC^N \\ 
        \| \Psi\,  \sqrt{\bmG} c  \|_{L^2}^2 = 1
    }}\quad \big\| (\Pi \scrL \Pi - \lambda I) \Psi c \big\|_{L^2}
\end{equation}
we have 
$$\setlength\abovedisplayskip{6pt}\setlength\belowdisplayskip{6pt}
    ( \min \sigma (\bmG) )\ \Gamma 
    \ \leq\ \lim_{M \to \infty} \widehat{ \kres } (\lambda; M, N) 
    \ \leq\ ( \max \sigma (\bmG) )\ \Gamma . 
$$

We will next show a stronger connection to residuals with respect to $\scrL$. A key step is deriving a natural normalization in~\cref{eq:GammaDef}, ensuring that $\Psi c$ is $L^2$-normalized. This naturally leads to minimizing over $c^*Gc=1$ in~\cref{eq:kres_hat_lambda}, as is done in~\cref{eq:kres_def} below.

\begin{algorithm}[t]
    \caption{Kernelized ResDMD as in \cite{kresdmd}}
    \label{alg:kresdmd}
    \begin{algorithmic}[1]
        \Require kernel $k : \Omega \times \Omega \to \bbC$, data points 
            $\left\{ x_i \right\}_{i=1}^M$, compression factor $r \leq M$,
            grid $\left\{ z_\nu \right\}_{\nu=1}^T \subset \bbC$,
            tolerance $\epsilon$ 
        \State Construct $\widehat{G} = (\, \frac{1}{M} k ( x_l, x_i ) \,)_{i, l = 1}^M$, 
        $\ \widehat{A} = (\, \frac{1}{M} k ( S(x_l), x_i ) \,)_{i, l = 1}^M$, 
        \mbox{$\ \widehat{J} = (\, \frac{1}{M} k ( S(x_l), S(x_i) ) \,)_{i, l = 1}^M$}
        \State Compute an eigendecomposition $\widehat{G} = Q \Sigma^2 Q^*$
        \State Let $\widetilde{\Sigma} = \Sigma [1:r, 1:r]$, $\widetilde{Q} = Q [:, 1:r]$ ($r$ largest singular values and vectors)
        \State Construct $\widehat{K} = 
            ( \widetilde{\Sigma}^\dagger \widetilde{Q}^* )
            \widehat{A}
            ( \widetilde{Q} \widetilde{\Sigma}^\dagger )$
        \For{$z_\nu$}
        \State Compute $\widehat{ U } =
            ( \widetilde{ \Sigma }^\dagger \widetilde{ Q }^* ) 
            \widehat{ J } ( \widetilde{ Q } \widetilde{ \Sigma }^\dagger )
            - \overline{z_\nu} \widehat{ K } - z_\nu \widehat{ K }^* + | z_\nu |^2 I$
        \State Compute $\widehat{ \kres } (z_\nu)=\sqrt{\xi}$, where $\xi$ is  
            the smallest eigenvalue of $\widehat{ U }$ 
        \EndFor
        \State \Return $\{ z_\nu \mid \widehat{\kres} (z_\nu) < \epsilon \}$
    \end{algorithmic}
\end{algorithm}

\subsection{A New Residual with the Desired Limit Behavior}\label{sec:kres}

To derive a computable estimate for the true residual $\big\| (\scrL - \lambda
I) \Psi c \big\|_{L^2}$, we first take a closer look at the approximation
(sub)space $\spn(\Sigma^2 Z^*)$ in~\cref{eq:kres_hat_lambda}. The key
construction which enables us to let $N\to\infty$ before $M\to \infty$ lies in
the compression factor $r$ from \cref{alg:kedmd}: Dropping the previous
assumption $r = \text{rank} ( \YX )$, we now fix $r \leq \text{rank} ( \YX )$
and consider the \emph{truncated} SVD 
\begin{equation}\setlength\abovedisplayskip{6pt}\setlength\belowdisplayskip{6pt}
    \label{eq:truncated_SVD}
    \tfrac{1}{\sqrt{M}} \YX \approx \widetilde{ Q } \widetilde{ \Sigma } \widetilde{ Z }^*, \quad
    \widetilde{ Q } \in \bbC^{M \times r},\ 
    \widetilde{ \Sigma } \in \bbC^{r\times r},\ 
    \widetilde{ Z }^* \in \bbC^{r \times N} . 
\end{equation}
Considering the $\Psi(x_i)$, $i=1,\ldots,M$, as $N$-dimensional data points, their empirical covariance matrix decomposes as $\frac1M \Psi_X^* \Psi_X = \widetilde{ Z } \widetilde{ \Sigma }^2 \widetilde{ Z }^*$. This implies that $\widetilde{ \Sigma }^2 = \widetilde{ Z }^* \left(\frac1M \Psi_X^* \Psi_X\right) \widetilde{ Z }$, which reads as discretized $L^2$-orthogonality, when expanded:
$$\setlength\abovedisplayskip{6pt}\setlength\belowdisplayskip{6pt}
\frac1M \sum_{i=1}^M \Big( \sum_{n=1}^N \psi_n(x_i) \widetilde{ Z }_{nj} \cdot \sum_{n=1}^N \psi_n(x_i) \widetilde{ Z }_{nj'} \Big) = \left\{ \begin{array}{ll}
    0 & j\neq j', \\
    \widetilde{ \Sigma }_{jj}^2 & j=j'.
\end{array}\right.
$$
Borrowing from statistical learning theory~\cite[p.~66]{statisticallearning}, the subspace spanned by the $L^2$-orthogonal observables $\smash{ \widetilde{\psi}_j = \sum_{n = 1}^N \widetilde{Z}_{n j}\, \psi_n }$, $j=1,\ldots,r$, has the largest variance when evaluated on the data points~$x_i$ (constrained to the columns of $\widetilde{ Z }$ being orthonormal vectors in $\bbC^N$).
The matrix $\widetilde{Z}^*$ represents the transformation from the dictionary space spanned by $\psi_1, \ldots, \psi_N$ to the
space spanned by the $r$ largest principal orthogonal components 
$ \widetilde{ \psi }_1, \ldots, \widetilde{ \psi }_r \subset L^2 $.
Note that we deliberately used $\frac1M$ instead of $\frac{1}{M-1}$ for the empirical covariance matrix, because on the one hand the principal orthogonal components are not affected by this change and on the other hand for large $M$ the difference vanishes.

The benefit of this truncation is that we \emph{decouple} the $N$ and $M$
limits and therefore allow performing the limit $N \to \infty$ \emph{before} $M
\to \infty$, which is necessary since a priori the kernel feature map may be
infinite-dimensional. Using the new dictionary 
\begin{equation}\setlength\abovedisplayskip{6pt}\setlength\belowdisplayskip{6pt}
    \label{eq:Psi_tilde}
    \smash{\widetilde{\Psi} = \big[ \widetilde{ \psi }_1 \mid \ldots \mid \widetilde{ \psi }_r \big]}
\end{equation}
and $M \geq r$ data points $x_i$, we construct the data matrix $\tYX = (\widetilde \psi_j(x_i))_{ij}$. 
Note that the $\widetilde{\psi}_j$, $j=1,\ldots,r$, are obtained the very same way irrespective of $N<\infty$ or~$N = \infty$.

Consider for a moment the purely formal limits $N \to \infty$ before $M \to
\infty$. Using the same thoughts as in preceding subsections, 
we notice that for a candidate eigenpair $(\lambda, h) \in
\bbC \times \spn \big\{\widetilde{ \psi }_1, \ldots, \widetilde{ \psi }_r \big\}$ for~$\scrL$, where $h = \tY u$ is encoded by the vector $u \in \bbC^r$, we have that
\begin{equation}\setlength\abovedisplayskip{6pt}\setlength\belowdisplayskip{6pt}
    \label{eq:kres_as_adjoints}
    \begin{split}
        \lim_{M \to \infty} \Big\| \frac{1}{\sqrt{M}} \left( \YY^* - \bar{\lambda} \YX^* \right)& \tYX u \Big\|_{\ell^2}^2 
        = \left\| \left(\, (\scrK \Psi)^* - \bar{\lambda} \Psi^* \,\right) h \right\|_{\ell^2}^2 \\
        &= \sum_{j=1}^{\infty} \left| \left\langle\ (\scrK - \bar{\lambda} I) \psi_j, h \ \right\rangle \right|^2 
        = \sum_{j=1}^{\infty} \left| \left\langle\ \psi_j, (\scrL - \lambda I) h \ \right\rangle \right|^2 . 
    \end{split}
\end{equation}
This suggests that one should define 
\begin{equation}\setlength\abovedisplayskip{6pt}\setlength\belowdisplayskip{6pt}
    \label{eq:kres_def}
    \kres (\lambda, u; r, M) = 
    \frac{1}{\sqrt{M}} \left\| \left( \YY^* - \bar{\lambda} \YX^* \right) \tYX u \right\|_{\ell^2} , \quad
    \kres (\lambda; r, M) = 
    \min_{u^* \widetilde{ \Sigma }^2 u = 1} \kres (\lambda, u; r, M) .
\end{equation}

\subsection{Convergence Theorem} To prove a convergence theorem for this residual, we require the following lemma.

\begin{lemma}
    \label{lem:kres}
    Let $k \colon \Omega \times \Omega \to \bbC$ be a Mercer kernel and $\Psi$
    be the dictionary of Mercer features $\psi_j = \mu_j \phi_j$, $j = 1, 2, \ldots$. 
    Let further $\widetilde{ \Psi }$ and
    $\kres$ be defined as in~\cref{eq:kres_def,eq:Psi_tilde}, respectively, for
    an $r \in \bbN$. Then 
    \begin{equation}\setlength\abovedisplayskip{6pt}\setlength\belowdisplayskip{6pt}
        \label{eq:kres_M_limit}
        \lim_{M \to \infty}\ \kres (\lambda; r, M)^2 
        = \min_{\substack{
            h \in \spn \{ \widetilde{ \psi }_1, \ldots, \widetilde{ \psi }_r \} \\ 
            \left\| h \right\|^{}_{L^2} = 1
        }}\ \sum_{j=1}^{\infty} \left| \left\langle\ 
            \psi_j, (\scrL - \lambda I) h 
        \ \right\rangle \right|^2 
        \quad \forall\ \lambda \in \bbC . 
    \end{equation}
\end{lemma}

\begin{proof}
See \cref{sec:appendixA}.
\end{proof}

\begin{theorem}
    Let $k$, $\Psi$, $\tY$, $\kres,\mu_j$ be as in Lemma~\ref{lem:kres}. 
    Then for all $\epsilon > 0$ there exists an $\bar{M} = \bar{M} (\epsilon) > 0$ such that 
    \begin{equation}\setlength\abovedisplayskip{6pt}\setlength\belowdisplayskip{6pt}
        \label{eq:kres_fixed_h}
        \kres (\lambda, u; r, M) < \mu_1\left\| \left( \scrL - \lambda I \right) \widetilde{
    \Psi }u \right\|_{L^2} + \epsilon\quad\forall \lambda \in \bbC,
        u \in \bbC^r,M>\bar{M}.
    \end{equation}
        Assume additionally that the RKHS generated by $k$ is dense in $L^2$. Then 
    \begin{equation}\setlength\abovedisplayskip{6pt}\setlength\belowdisplayskip{6pt}
        \label{eq:lower_bound}
        \lim_{r \to \infty} \lim_{M \to \infty} \kres (\lambda; r, M) \,\leq\, 
        \mu_1 \min_{\| h \|_{L^2} = 1} 
        \left\| \left( \scrL - \lambda I \right) h \right\|_{L^2} 
        \quad \forall\, \lambda \in \bbC . 
    \end{equation}
    Moreover, the left-hand side of \cref{eq:lower_bound} is strictly greater
    than $0$ whenever the right-hand side is greater than $0$, but there is no
    uniform lower bound of the form constant multiplied by $\cdot \min_{\| h \|^{}_{L^2} = 1}
    \left\| \left( \scrL - \lambda I \right) h \right\|_{L^2}$. 
\end{theorem}

\begin{proof}
    Let $\lambda \in \bbC$, $h = \tY u$. Without loss let $\| h \|^{}_{L^2} = 1$
    (otherwise simply rescale). Noting that $\psi_j = \mu_j \phi_j$ are the
    Mercer features, we can rewrite
    $$\setlength\abovedisplayskip{6pt}\setlength\belowdisplayskip{6pt}
        \sum_{j=1}^{\infty} \left| \left\langle\ 
        \psi_j, (\scrL - \lambda I) h 
        \ \right\rangle \right|^2 
        \leq \mu_1^2 \sum_{j=1}^{\infty} \left| \left\langle\ 
            \phi_j, (\scrL - \lambda I) h 
        \ \right\rangle \right|^2
        = \mu_1^2 \left\| (\scrL - \lambda I) h \right\|_{L^2}^2
    $$
    by Parseval's identity since the $\mu_j$ are ordered by decreasing
    magnitude. 

    The proof of \cref{eq:kres_M_limit} yields \cref{eq:kres_fixed_h}. 
    Moreover, if the RKHS generated by $k$ is dense in $L^2$, then
    \cref{eq:lower_bound} follows from \cref{eq:kres_fixed_h} and
    \cref{eq:kres_M_limit}. The final claim of the theorem follows from the fact
    that if $(\scrL - \lambda) h \neq 0$, then there exists a $\phi_j$ such that
    $\left\langle \phi_j, ( \scrL - \lambda I ) h \right\rangle > 0$ because the
    $\phi_j$'s form a complete orthonormal family in $L^2$. However, $\mu_j \to
    0$ as $j \to \infty$ so a $C > 0$ with $C \cdot \min_{\| h \|^{}_{L^2} = 1}
    \left\| \left( \scrL - \lambda I \right) h \right\|^{}_{L^2} \leq \lim_{r
    \to \infty}^{} \lim_{M \to \infty}^{} \kres (\lambda; r, M)$ does not exist. 
\end{proof}

\begin{remark}
    If $\int_\Omega | k(x, x) |^p \,dx < \infty$ for a $p \geq 1$ then by the
    classical theory on Schatten-$p$-class operators~\cite{simon2005trace} and
    the fact that $k$ is symmetric and positive definite, $\mu_1^p \leq
    \int_\Omega | k(x, x) |^p \,dx$. 
\end{remark}

\paragraph{Interpretation}
The theorem does not give an explicit method for computing the pseudospectrum on $L^2$, but instead provides a \emph{necessary} condition for an eigenpair to be $\epsilon$-pseudospectral—allowing one to reject spurious candidates. Since the identity is not a Fredholm integral operator, the multipliers $\mu_j$ must decay to zero. As a result, no sufficient condition on $L^2$ can be obtained from the theorem. Nevertheless, this offers an operator-theoretic perspective on \cref{alg:kresdmd}. In \cite{kresdmd}, the residual $\smash{\widehat{\kres}}$ was introduced ad hoc, and our theorem provides a reason for the observed low residuals.

There is a subtle difference between
$\widehat{ \kres }$ and $\kres$. In the regime $M \leq N < \infty$, if we chose
$r = \text{rank} ( \YX )$ and made the substitution $u = Z^* c$, $c \in \bbC^N$
in \cref{eq:kres_as_adjoints}, we would have $\|\frac{1}{M} ( \YY^* -
\bar{\lambda} \YX^* ) \YX c \|_{\bbC^N}^2$. This is (up to conjugation of
$\lambda$) identical to the term to be minimized in \cref{eq:kres_hat_lambda}.
The truncation and removal of $\smash{\widetilde{ Z }^*}$ in \cref{sec:kres}
simply fixes a basis $\smash{\{ \widetilde{ \psi }_1, \ldots, \widetilde{ \psi
}_r \}}$ of the $r$-dimensional subspace of $L^2$ over which we minimize in
\cref{eq:kres_M_limit}, and allows the $N$ and $M$ limits to occur in the right
order. Combined with the ``natural'' normalization, these subtle changes
unlocked a functional-analytic interpretation.

\subsection{Computation}\label{sec:kres_computation}

We conclude by demonstrating how to compute $\kres (\lambda)$. We have 
$$\setlength\abovedisplayskip{6pt}\setlength\belowdisplayskip{6pt}
    \kres (\lambda)^2
    = \min_{u^* \widetilde{ \Sigma }^2 u = 1} \kres (\lambda, u)^2 = \min_{u^* \widetilde{ \Sigma }^2 u = 1} \left( \tYX u \right)^* \left( 
        \widehat{ J } 
        - \lambda \widehat{ A } 
        - \bar{\lambda} \widehat{ A }^* 
        + | \lambda |^2 \widehat{ G }
    \right) \left( \tYX u \right) .
$$
Letting $w = \widetilde{ \Sigma } u$ and noting that by assumption $\widetilde{
\Sigma }_{i i} \neq 0$ for all $1 \leq i \leq r$, this becomes 
$$\setlength\abovedisplayskip{6pt}\setlength\belowdisplayskip{6pt}
    \min_{w^* w = 1} w^* \left( 
        \widetilde{ Q }^* \widehat{ J } \widetilde{ Q }
        - \lambda \widetilde{ Q }^* \widehat{ A } \widetilde{ Q }
        - \bar{\lambda} \widetilde{ Q }^* \widehat{ A }^* \widetilde{ Q }
        + | \lambda |^2 \widetilde{ \Sigma }^2
    \right) w.
$$
With 
$
    \widetilde{ J } = \widetilde{ Q }^* \widehat{ J } \widetilde{ Q }, 
    \widetilde{ A } = \widetilde{ Q }^* \widehat{ A } \widetilde{ Q },
    \widetilde{ G } = \widetilde{ \Sigma }^2, 
$
it follows that $\kres (\lambda)^2$ can be computed from the smallest
eigenvalue of 
\begin{equation}\setlength\abovedisplayskip{6pt}\setlength\belowdisplayskip{6pt}
    \label{eq:kres_3}
    \widetilde{ U }  =
        \widetilde{ J } 
        - \lambda \widetilde{ A } 
        - \bar{\lambda} \widetilde{ A }^*
        + | \lambda |^2 \widetilde{ G } 
        \;\ \in\ \bbC^{r \times r} . 
\end{equation}
\cref{alg:kresdmd_correct} summarises the procedure.

\begin{algorithm}[t]
    \caption{Modified kernel ResDMD with an operator-theoretic interpretation}
    \label{alg:kresdmd_correct}
    \begin{algorithmic}[1]
        \Require kernel $k : \Omega \times \Omega \to \bbC$, data points 
            $\{ x_i \}_{i=1}^M$, compression factor $r \leq M$,
            grid $\left\{ z_\nu \right\}_{\nu=1}^T \subset \bbC$,
            tolerance $\epsilon$
        \State Construct $\widehat{G} = (\, \frac{1}{M} k ( x_l, x_i ) \,)_{i, l = 1}^M$, 
        $\ \widehat{A} = (\, \frac{1}{M} k ( S(x_l), x_i ) \,)_{i l = 1}^M$, 
        \mbox{$\ \widehat{J} = (\, \frac{1}{M} k ( S(x_l), S(x_i) ) \,)_{i l = 1}^M$}
        \State Compute an eigendecomposition $\widehat{G} = Q \Sigma^2 Q^*$
        \State Let $\widetilde{\Sigma} = \Sigma [1:r, 1:r]$, $\widetilde{Q} = Q [:, 1:r]$ ($r$ largest singular values and vectors)
        \State Construct $\widetilde{J} = \widetilde{Q} \widehat{J} \widetilde{Q}$, 
            $\ \widetilde{A} = \widetilde{Q} \widehat{A} \widetilde{Q}$, 
            $\ \widetilde{G} = \widetilde{\Sigma}^2$
        \For{$z_\nu$}
        \State Compute $\widehat{ U } = \widetilde{J} - z_\nu \widetilde{A} 
            - \overline{z_\nu} \widetilde{ A }^* + | z_\nu |^2 \widetilde{G}$
        \State Compute $\kres (z_\nu)=\sqrt{\xi}$, where $\xi$ is the smallest eigenvalue of $\widehat{ U }$
        \EndFor
        \State \Return $\left\{ z_\nu \mid \kres (z_\nu) < \epsilon \right\}$
    \end{algorithmic}
\end{algorithm}

\section{Numerical Experiments}\label{sec:numerics}

We present two experiments illustrating the algorithms from the previous sections. The first uses a system with a known analytical structure to highlight the risks of applying dynamic mode decomposition without error quantification. The second, using real-world protein-folding data, shows how the algorithms perform in practice. Code for all examples is provided in \cite{script}.

\subsection{A Blaschke Product}\label{sec:blaschke}

We consider a family of (complex) analytic
circle maps
\begin{equation}\label{eq:simple_blaschke}\setlength\abovedisplayskip{6pt}\setlength\belowdisplayskip{6pt}
    S : \bbT \to \bbT,\quad z \mapsto z \frac{z - \mu}{1 - \bar{\mu} z}, 
\end{equation}
for $\mu \in \bbD$, with $\bbD$ the open unit disk. The map $S$ is a two-to-one
map on the unit circle $\bbT = \partial \bbD$ and can be analytically extended
to the open annulus $\bbA_r = \left\{ z \in \bbC \mid r < |z| < r^{-1} \right\}$ for
any $r\in [|\mu|, 1)$. The spectrum of the Frobenius--Perron operator $\scrL$ associated to the map in \eqref{eq:simple_blaschke}
has been studied analytically in \cite{SlipantschukNonlin} and for general Blaschke maps in \cite{Slipantschuk}. It was shown that on the
following function space, which is densely and continuously embedded in
$L^2 (\bbT)$, $\scrL$ is compact and has a simple spectrum: 
The space $H^2(\bbA_r)$ of holomorphic functions on $\bbA_r$ which can be extended
to functions that are square integrable on $\partial \bbA_r$. This is known as a \emph{Hardy--Hilbert space}
with inner product 
$$\setlength\abovedisplayskip{6pt}\setlength\belowdisplayskip{6pt}
    {\left\langle f, g \right\rangle}_{H^2 (\bbA_r)}
    = \left[ \lim_{\rho \searrow r }\ \frac{1}{2 \pi} \int_0^{2\pi} f (\rho e^{i \theta}) \cdot \overline{g (\rho e^{i \theta})} \, d\theta \right]
    + \left[ \lim_{\rho \nearrow r^{-1}}\ \frac{1}{2 \pi} \int_0^{2\pi} f (\rho e^{i \theta}) \cdot \overline{g (\rho e^{i \theta})} \, d\theta \right] . 
$$
It is not hard to see that $e_n (z) = z^n / \sqrt{r^{2n} + r^{-2n}}$ is an
orthonormal basis of $H^2(\bbA_r)$. We do not make much use of the structure of $H^2
(\bbA_r)$ at first, aside from taking note that $H^2 (\bbA_r)$ is  isomorphic to a
subspace of $L^2 (\bbT)$. A theorem of \cite{Slipantschuk} states that $\left.
\scrL \right|_{H^2 (\bbA_r)}$ is compact (in fact, Hilbert--Schmidt) and
\begin{equation}\label{eq:blaschke_spectrum}\setlength\abovedisplayskip{6pt}\setlength\belowdisplayskip{6pt}
    \sigma \left( \left. \scrL \right|_{H^2(\bbA_r)} \right) 
    = \sigma_p \left( \left. \scrL \right|_{H^2(\bbA_r)} \right) \cup \{0\}
    = \left\{ \mu^n \mid n \in \bbN_0 \right\} \cup \left\{ \overline{\mu}^{\,n} \mid n \in \bbN_0 \right\} 
    \cup \{ 0 \} . 
\end{equation}
As a consequence of the embedding $H^2 (\bbA_r) \hookrightarrow L^2 (\bbT)$ we have
$\sigma_p ( \left. \scrL \right|_{H^2 (\bbA_r)} ) \subset \sigma_p ( \left. \scrL
\right|_{L^2 (\bbT)} )$.

\subsubsection*{An Initial Numerical Experiment}\label{sec:blaschke_L2} 

We consider the Blaschke product map \cref{eq:simple_blaschke} with $\mu = \frac{3}{4} e^{i
\pi / 4}$. In \cref{fig:blaschke_L2}, the spectrum \cref{eq:blaschke_spectrum} 
of $\left. \scrL \right|_{H^2(\bbA_r)}$
is shown by black dots.  For the EDMD approximation (cf.\ \cref{sec:edmd}), we use a Fourier basis $\psi_n (\theta) = e^{i \pi
n \theta}$, $n = -20, \ldots, 20$ (i.e.\ $N=41$), as a dictionary  and $M = 1000$ equidistant quadrature nodes. The spectrum of the resulting matrix $L$ is shown in \cref{fig:blaschke_L2} by orange crosses,  
matching the spectrum of $\left. \scrL \right|_{H^2(\bbA_r)}$
 (visually) exactly.  This is because (when enough quadrature nodes are used)
 $\sigma(L)$ converges to $\sigma_p(\scrL \left|_{H^2(\bbA_r)}\right.)$ exponentially fast as~$N$ increases \cite{dmdanalytic, edmdexpanding}.

\begin{figure}[h]
    \centering
        \centering
        \includegraphics[trim={0 0 0 1cm},clip,width=0.5\textwidth]{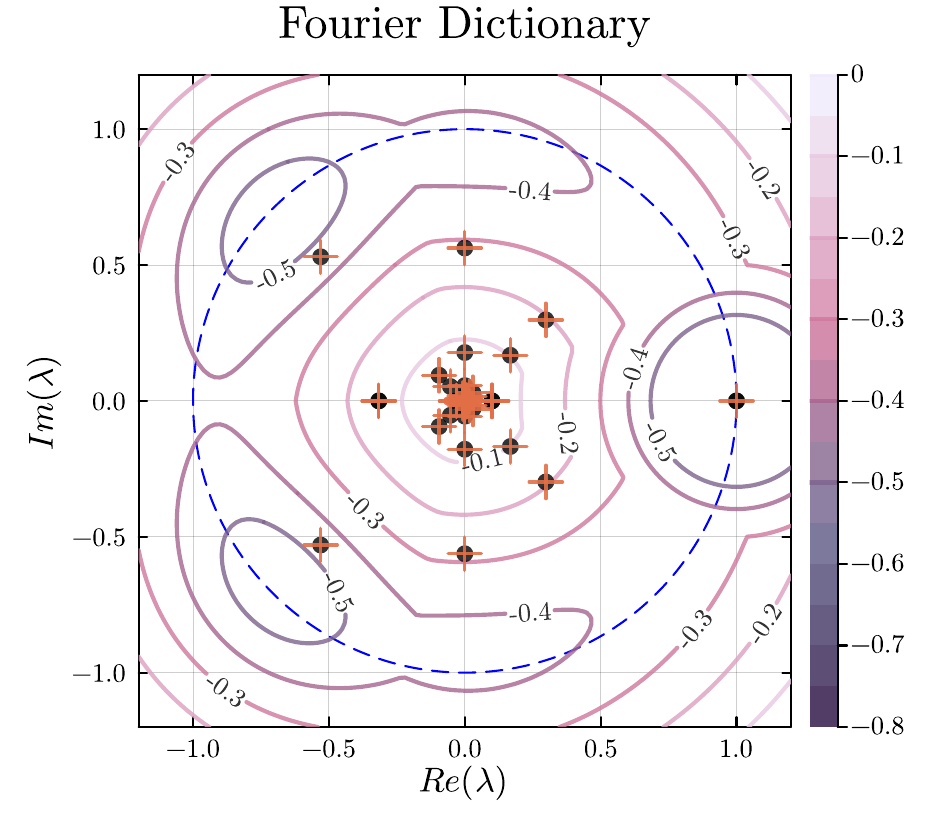}
    \caption{Spectrum of $\left.\scrL \right|_{H^2(\bbA_r)}$ (black dots), 
        spectrum of the EDMD matrix $L$ (orange crosses), residuals
        calculated using \cref{alg:resdmd} (contours are logarithmically scaled).}
        \label{fig:blaschke_L2}
\end{figure}

Since the Fourier basis is orthonormal on $\smash{L^2(\bbT)}$, we have
$G = I$ so that $L = K^*$ (where $K$ denotes the EDMD matrix approximating $\scrK$ from
\cref{sec:edmd}). Additionally, $\sigma(L)$ is symmetric about the real axis so $\sigma
(L) = \sigma (K)$.  It therefore seems tempting to use the ResDMD \cref{alg:resdmd} to compute 
pseudospectra of $\scrK \left|_{L^2(\bbT)}\right.$ and thus check for the reliability of the spectrum of $L$.
The residuals resulting from \cref{alg:resdmd} are shown by contour lines 
in \cref{fig:blaschke_L2}.  Clearly, none of the eigenvalues coincides with a local
minimum of the residual function.  The reason for this seemingly contradictory result
is that the point spectrum of the Koopman operator $\scrK$ on $L^2(\bbT)$ is $\sigma_p ( \left.
\scrK \right|_{L^2(\bbT)} )=\{1\}$, since the Blaschke product map \cref{eq:simple_blaschke}
is mixing. As a consequence, the eigenvalues $\mu^n,\bar\mu^n$ of $\scrL \left|_{H^2(\bbA_r)}\right.$
lie in the residual spectrum of $\left.\scrK \right|_{L^2(\bbT)}$.

\subsubsection*{ResDMD on the dual of $H^2 (\bbA_r)$}\label{sec:blaschke_hardy}

Nonetheless, we can still use ResDMD to verify the computed eigenvalues, provided we use the correct space. Since
$\smash{\left. \scrL \right|_{H^2 (\bbA_r)}}$ is compact and $\smash{\mu^n \in \sigma
( \left. \scrL \right|_{H^2 (\bbA_r)} )}$ for $n > 0$, its dual operator, identified
with the Koopman operator considered on $H^2(\bbA_r)^*$, is compact by Schauder's
theorem~\cite{schaudertheorem} and $\mu^n$ is contained in its spectrum. Here, $H^2 (\bbA_r)^*$ denotes the (Banach space) dual of $H^2(\bbA_r)$ equipped with the $L^2$ norm. We therefore need to do ResDMD on the larger space $H^2 (\bbA_r)^*$.

The space $H^2(\bbA_r)^*$ is isometrically isomorphic to the direct sum
$\scrX_r := H^2 (\bbD_{r }) \oplus H^2 (\bbD_{r^{-1}}^\infty)$ \cite{Slipantschuk}, where
$\bbD_{r} = \left\{ z \in \bbC\mid |z| < r \right\}$ and $H^2
(\bbD_{r^{-1}}^\infty)$ is the set of functions holomorphic on
$\smash{\hat{\mathbb{C}} \setminus \overline{\bbD_{r^{-1}}}}$ which are square
integrable on the boundary $\partial \bbD_{r^{-1}}$ and vanish at infinity. The
space $\scrX_r$ is endowed with
the inner product
\begin{equation}\setlength\abovedisplayskip{6pt}\setlength\belowdisplayskip{6pt}
    \langle f,g\rangle_{\scrX_r} = \sum_{n=-\infty}^{\infty} \overline{c_n (f)} c_n (g)\ r^{2 |n|} ,
\end{equation}
where $c_n (f)$
is the $n$th Fourier coefficient of~$f$. The triple $H^2
(\bbA_r) \subset L^2 (\bbT) \subset H^2 (\bbA_r)^* \simeq \scrX_r$ is known as a Gelfand
triple or rigged Hilbert space. In particular, the space $\scrX_r$ is strictly
larger than $L^2 (\bbT)$, forming a space of distributions (generalized
functions).

We now invoke \cref{alg:resdmd} on the dual
Hardy--Hilbert space $\scrX_r = H^2 (\bbD_{r}) \oplus H^2 (\bbD_{r^{-1}}^\infty)$
(for e.g.~$r=|\mu|$), i.e.\ we approximate $\bmG$ by $G=(G_{jl})_{jl}$ by  
$$\setlength\abovedisplayskip{6pt}\setlength\belowdisplayskip{6pt}
     \left\langle \psi_j, \psi_l \right\rangle_{\scrX_r} \approx G_{j l}
    = \sum_{n=-(N-1)/2}^{(N-1)/2} \overline{c_n (\psi_j)} c_n (\psi_l)\ r^{2 |n|},
$$
and analogously for the matrices $A$ and $J$. The resulting residuals
are shown in \cref{fig:blaschke_other}
(left). Here, in contrast to Figure~\ref{fig:blaschke_L2}, the residuals reproduce the true
point spectrum of $\scrK \left|_{H^2 (\bbA_r)^*}\right.$ (which, as noted above, is the same as $\sigma_p\left(\scrL\left|_{H^2(\bbA_r)}\right.\right)$). 

\begin{figure}[h]
    \centering
    \begin{subfigure}{0.49\textwidth}
        \centering
        \includegraphics[width=\textwidth]{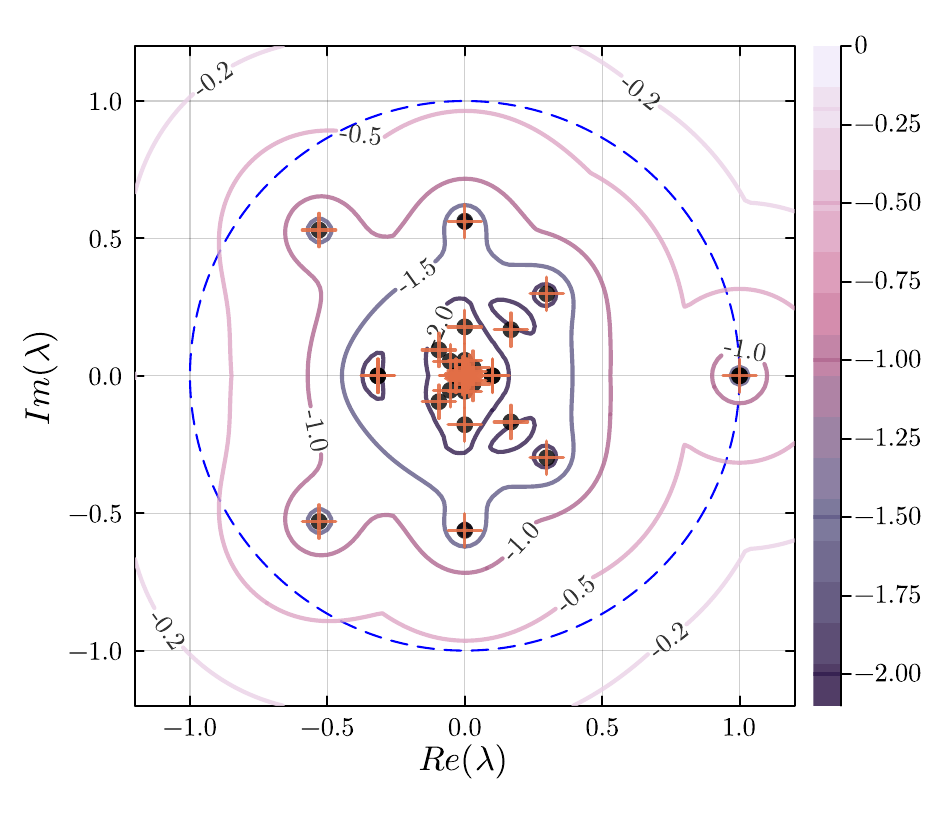}
    \end{subfigure}
    \hfill
    \begin{subfigure}{0.49\textwidth}
        \centering
        \includegraphics[width=\textwidth]{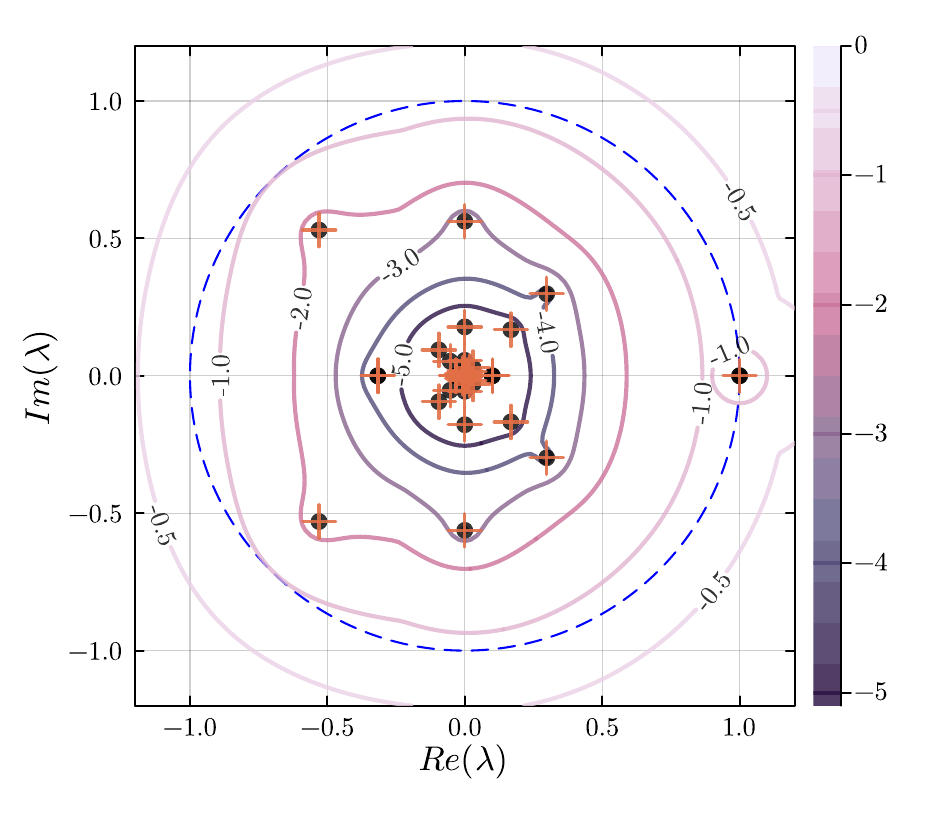}
    \end{subfigure}
    \caption{
        Left: residuals calculated using \cref{alg:resdmd} with the inner product on $H^2 (\bbA_r)^*$.
        Right: spectrum of $\widehat{ K }$ with residuals calculated using 
        \cref{alg:kresdmd_correct}, 
        In both cases, black: true spectrum of $\smash{\left. \scrL \right|_{H^2(\bbA_r)}}$, 
        orange: spectrum of the approximated matrix (left: $L$ using 
        the inner product on $H^2 (\bbA_r)^*$, right: $\widehat{ K }$). Contour lines are 
        \emph{logarithmically scaled}. 
    }\label{fig:blaschke_other}
\end{figure}

\subsubsection*{The Need for Adjoint Methods}

It is worth taking a moment to reconsider what has happened, as it is quite
unintuitive. When we shrunk the domain of $\scrL$ from $L^2(\bbT)$ to $H^2(\bbA_r)$, we removed all elements from
the spectrum which are not in the point spectrum. Correspondingly in the dual,
we enlarged the domain of $\scrK$ from $L^2(\bbT)$ to $H^2(\bbA_r)^*$, so that the residual spectrum 
vanished while some of its points \emph{became} point
spectrum.

This example illustrates the risks of using algorithms like EDMD ``as is''. The EDMD
matrix $K$ --- which was conceived as an approximation of $\smash{\left. \scrK
\right|_{L^2 (\bbT)}}$ --- captures eigenmodes which \emph{do not lie in
$L^2 (\bbT)$}. In order to make sense of the numerics, the function space had to
be very carefully enlarged. However, if one were to use $\scrL$ instead of
$\scrK$, one would not have needed this. This highlights the often overlooked
importance of developing numerical methods for both $\scrK$ \emph{and} $\scrL$,
and underlines the original motivations of this paper, which were to develop a
form of residual-based error control for $\scrL$. 

Indeed, when applying \cref{alg:kresdmd_correct} using a Gaussian kernel $k(w,
z) = \exp ( - \| w - z \|^2 / c^2 )$ with parameter $c^2 = 0.01$ --- which
generates a reproducing kernel Hilbert space similar to $H^2
(\bbA_r)$~\cite{gaussrkhs} --- one is left with similar results as when
applying \cref{alg:resdmd} in the space~$H^2 (\bbA_r)^*$,
cf.~\cref{fig:blaschke_other} (right). However, applying
\cref{alg:kresdmd_correct} requires no prior knowledge of the system. It is here
that we harvest the benefits of having numerical methods for both $\scrK$
and~$\scrL$; the process of delicately expanding the function space using
significant prior knowledge reduced to simply testing a different kernel
function.

\subsubsection*{A Generalized Heuristic}

The preceding result highlights the subtle -- yet crucial -- relationship
between the spectrum of $\scrK$ and the space on which $\scrK$ is considered.
To inspect the effect of smoothness on the analysis, we employ
\cref{alg:resdmd} again on the fractional Sobolev spaces $H^s(\bbT)$, $s\in\bbR$,
which carry the inner product 
$$\setlength\abovedisplayskip{6pt}\setlength\belowdisplayskip{6pt}
    \left\langle f, g \right\rangle_{H^s(\bbT)}
    = \int \left( 1 + | \xi |^2 \right)^s\ \overline{\scrF f (\xi)}\, \scrF g (\xi) \, d \xi,
$$
where $\scrF f$ is the Fourier transform of $f$. For $s\in\bbN$, $H^s(\bbT)$ is the
space of functions $f \in L^2(\bbT)$ whose derivatives of order up to $s$ are also in
$L^2(\bbT)$. In our case (of the circle as the domain), this inner product reduces to a weighted sum
of the Fourier coefficients 
$$\setlength\abovedisplayskip{6pt}\setlength\belowdisplayskip{6pt}
    \left\langle f, g \right\rangle_{H^s (\bbT)}
    = \sum_{n = -\infty}^\infty \left( 1 + | n |^2 \right)^s\ \overline{c_n (f)}\, c_n (g) . 
$$
Due to the embedding $H^{s'}(\bbT) \subset H^s(\bbT)$ for $s \leq s'$ 
\cite{besselpotential}, the fractional Sobolev spaces provide a way to
parametrically shrink or enlarge the space by restricting to function(al)s with
a prescribed level of smoothness. The residuals resulting from \cref{alg:resdmd}
on $H^2(\bbT)$, shown in
\cref{fig:blaschke_sobolev}, depend continuously and monotonically on the
parameter~$s$. 

\begin{figure}[h]
    \centering
    \begin{subfigure}{0.49\textwidth}
        \centering
        \includegraphics[width=\textwidth]{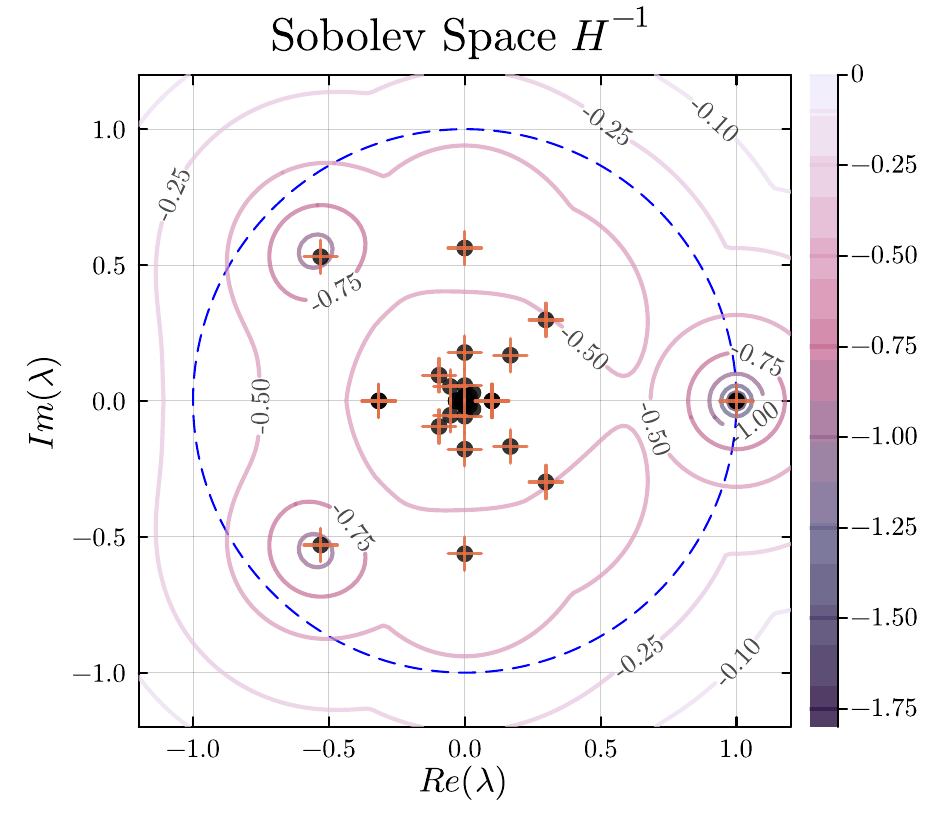}
    \end{subfigure}
    \hfill
    \begin{subfigure}{0.49\textwidth}
        \centering
        \includegraphics[width=\textwidth]{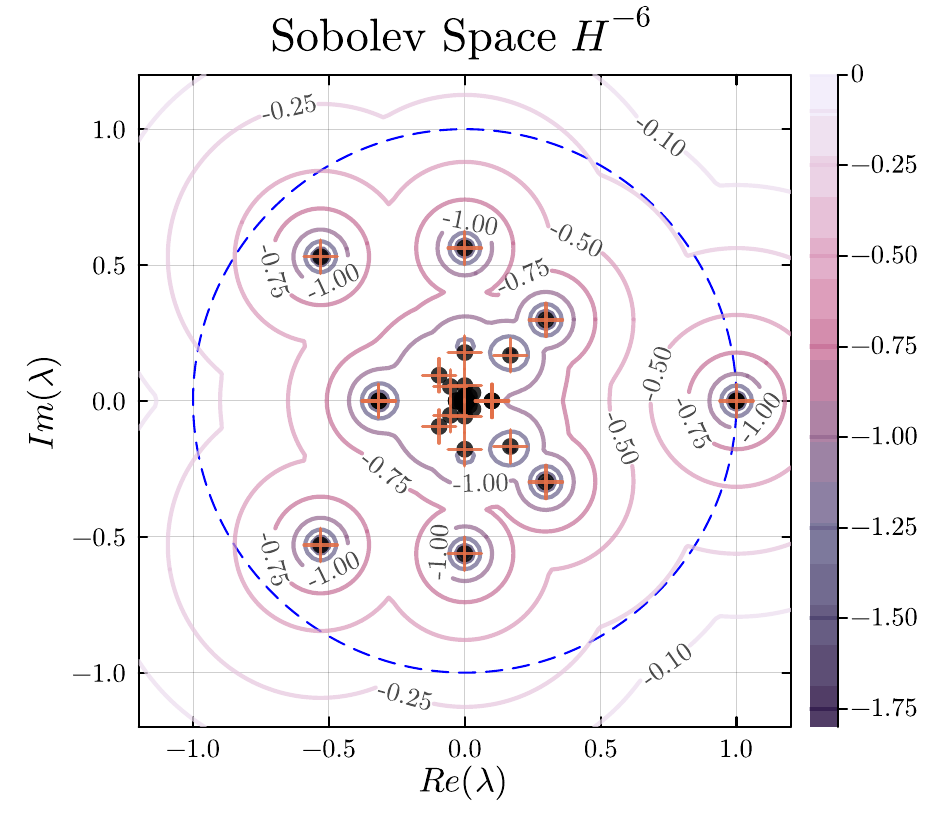}
    \end{subfigure}
    \caption{
        Residuals calculated using \cref{alg:resdmd} on the 
        fractional Sobolev spaces $H^s(\bbT)$. 
        Left: On $H^{-1}(\bbT)$, right: on $H^{-6}(\bbT)$. 
        In both cases, black: true spectrum of $\left. \scrL \right|_{H^2(\bbA_r)}$, 
        orange: spectrum of $L$. Contour lines are 
        logarithmically scaled.
    }\label{fig:blaschke_sobolev}
\end{figure}

\subsubsection*{Normality}  

In \cref{fig:blaschke_sobolev}, it seems as if we decrease~$s$ (i.e.\ enlarge the space), then
$\left. \scrK \right|_{H^s(\bbT)}$ becomes ``more normal'' (since the level sets of the residual
function become ``more disc like'').  And indeed, the
operator norm $\| \scrK \scrL - \scrL \scrK \|_{H^s(\bbT)}$ 
decreases with decreasing~$s$ as shown in~\cref{fig:deviation_from_normality}. 
In order to approximate $\| \scrK \scrL - \scrL \scrK \|_{H^s(\bbT)}$ we first
$H^s(\bbT)$-orthonormalize the Fourier dictionary and then compute the
EDMD matrices $K$ and $L$ using the $H^s(\bbT)$ inner product as usual. Note that
$\| L K - K L \|_{\bbC^N} \approx 
\| \Pi ( \scrL \Pi \scrK - \scrK \Pi \scrL ) \Pi \|_{H^s(\bbT)}$, where 
 $\Pi$ is the orthogonal projector onto the span of the dictionary.

\begin{figure}[h]
    \centering
    \includegraphics[width=0.5\linewidth]{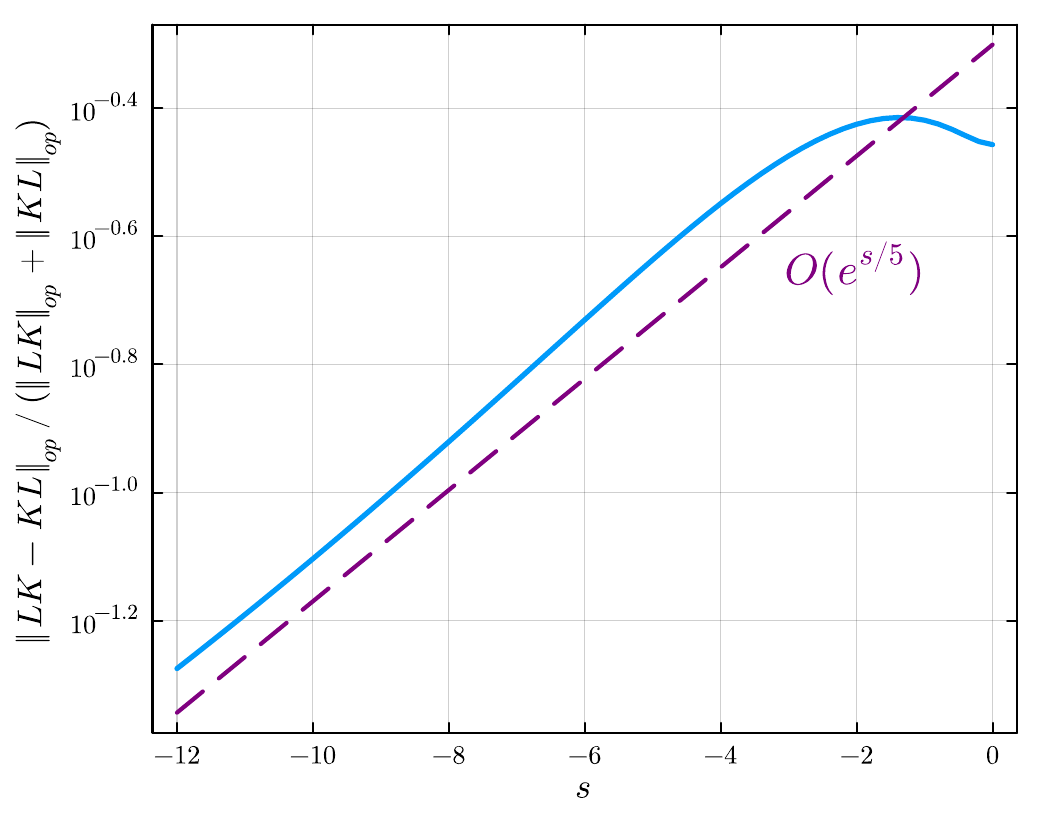} 
    \caption{``Deviation from normality'' of the operator $\left. \scrK \right|_{H^s(\bbT)}$ in dependence of~$s$.}
    \label{fig:deviation_from_normality}
\end{figure}

The reason for this phenomenon is the following: The ``infinite matrix'' 
representation of $\scrK$ in the Fourier basis has the form 
$$\setlength\abovedisplayskip{6pt}\setlength\belowdisplayskip{6pt}
    \left\langle \scrK \frac{e_i}{\| e_i \|_{H^s(\bbT)}^{}},\, \frac{e_j}{\| e_j \|_{H^s(\bbT)}^{}} \right\rangle_{H^s(\bbT)}
    = \left( \frac{1 + | j |^2}{1 + | i |^2} \right)^{s / 2} c_j ( \scrK e_i )
$$
where $e_i : z \mapsto z^i$ is the $i$-th Fourier mode. It is known~\cite{Slipantschuk} that 
$$\setlength\abovedisplayskip{6pt}\setlength\belowdisplayskip{6pt}
    \left(\, c_j ( \scrK e_i ) \,\right)_{i, j \in \bbZ} 
    = \left[\begin{smallmatrix}
        \ddots  &       &       &   \vdots   &       &          &       &       &            \\
                &   *   &   *   &      *     &       &          &       &       &            \\
                &       &   *   &      *     &       &          &   0   &       &            \\
                &       &       &      *     &       &          &       &       &            \\       
                &       &       &            &   1   &          &       &       &            \\
                &       &       &            &       &     *    &       &       &            \\
                &       &   0   &            &       &     *    &   *   &       &            \\
                &       &       &            &       &     *    &   *   &   *   &            \\
                &       &       &            &       &  \vdots  &       &       &  \ddots    \\
    \end{smallmatrix}\right]
$$
from which it is clear that (since $| j | > | i |$ on the nonzero off-diagonal entries) 
 the negative exponent $s / 2$ serves to suppress the off-diagonal elements, so that
 $\left(\, c_j ( \scrK e_i ) \,\right)_{i, j \in \bbZ}$ becomes closer to a diagonal
 matrix for smaller~$s$.
 This result is due to the expansivity of the mapping; 
when the map is contractive, the off-diagonal entries of the ``infinite matrix'' representation of $\scrK$ are 
suppressed by large positive~$s$.

While the structure of the mapping proved useful as a demonstration, 
the above methodology is more general and not tailored to the specific problem
considered here. The use of $H^s$ spaces for dynamic mode decompositions may be
advisable whenever the computation of the Fourier transform is feasible, and one
suspects that the eigenfunctions may be smooth. This can further be exploited
via kernels, as in \cref{alg:kresdmd_correct}. For example, the kernel
$$\setlength\abovedisplayskip{6pt}\setlength\belowdisplayskip{6pt}
    k(w, z) = \int \left( 1 + | \xi |^2 \right)^s \exp (i (w - z) \cdot \xi)\, d \xi 
$$
generates $H^s(\mathbb{R}^d)$~\cite{RKHS}. This is left to future work.

\begin{figure}[!ht]
    \centering
    \begin{subfigure}{0.49\textwidth}
        \centering
        \includegraphics[width=\textwidth]{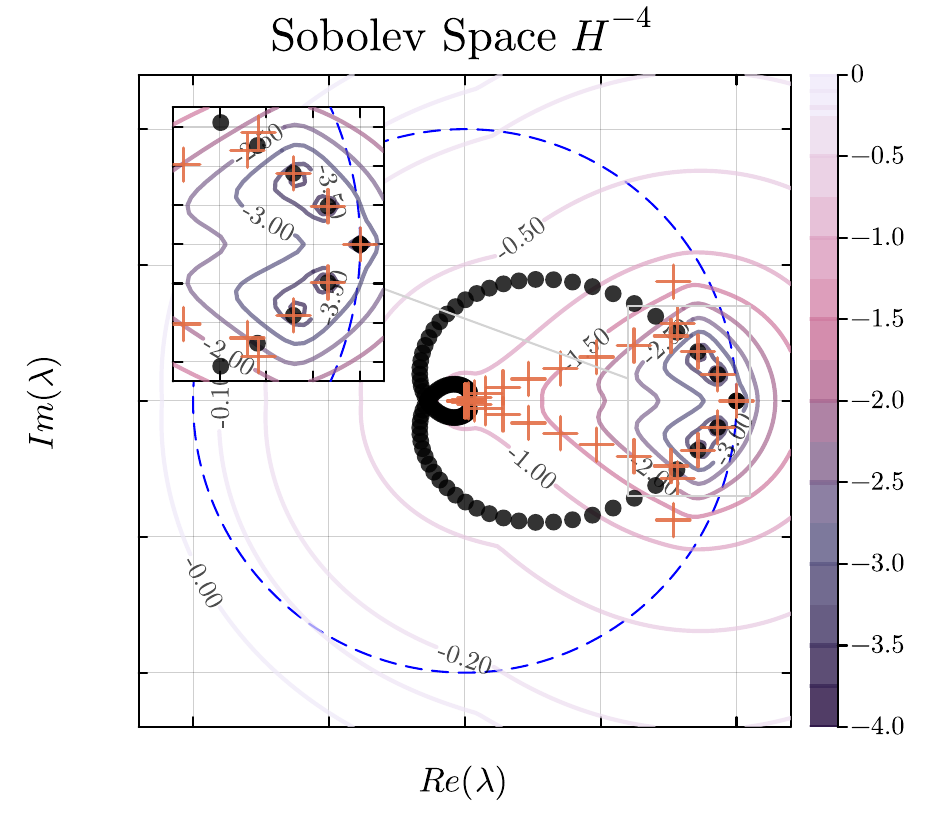}
    \end{subfigure}
    \hfill
    \begin{subfigure}{0.49\textwidth}
        \centering
        \includegraphics[width=\textwidth]{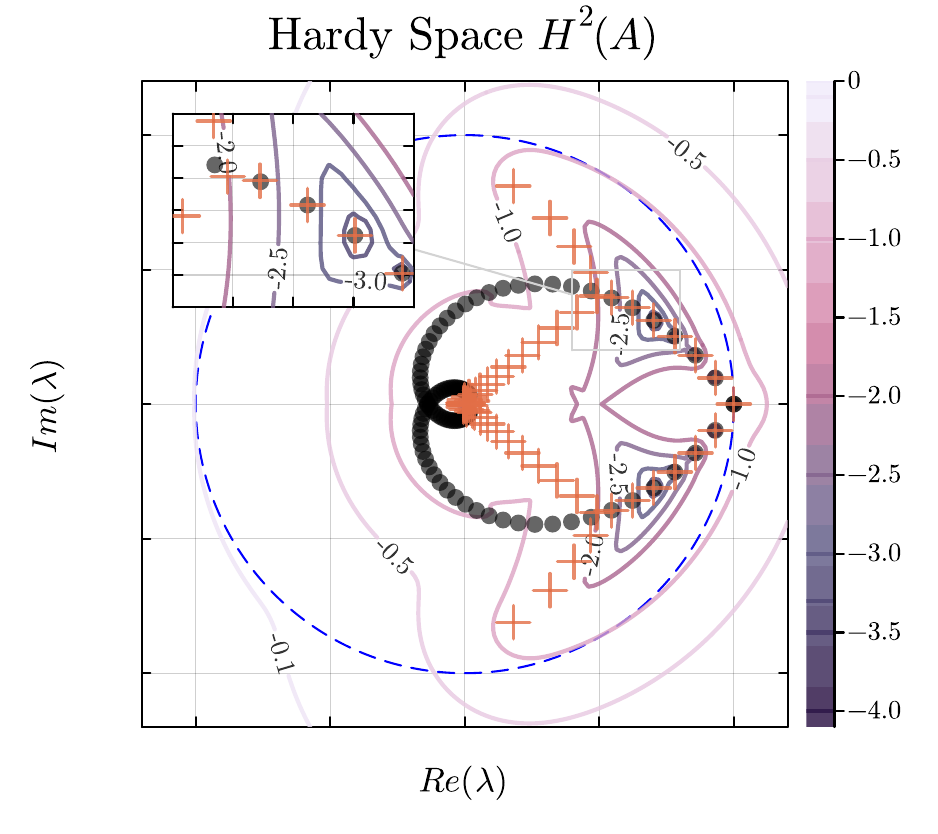}
    \end{subfigure}
    \caption{
        Spectrum of $L$ (and hence also of $K = L^*$) with residuals calculated using \cref{alg:resdmd} in the 
        fractional Sobolev space $H^{-4}$ (left) and in the Hardy--Hilbert space $H^2(\bbA_r)^*$
        with $r=0.755$ (right).  
        In both cases, black: true spectrum of $\left. \scrL \right|_{H^2(\bbA_r)}$, 
        orange: spectrum of the approximating matrix $L$. We have chosen $M = 10000$ and $N=50$ 
        in \cref{alg:resdmd}. Contour lines are 
        \emph{logarithmically scaled}, i.e.~they show the approximated 
        $\epsilon$-pseudospectrum for e.g. 
        $\epsilon = 10^{-p}$ for some $p \in (0,4)$.
        See \cref{sec:another_blaschke} for analysis. 
    }\label{fig:another_blaschke}
\end{figure}

\subsection{Another Blaschke Product}\label{sec:another_blaschke}
We shall next pay attention to another aspect of the performance of the above
algorithms. One of the main purposes of computing pseudospectra or
$\epsilon$-pseudospectra is the detection of spurious eigenvalues in numerical
computations. For the chosen Blaschke map example in \cref{eq:simple_blaschke},
the eigenvalues of the EDMD matrix $K$ computed using a Fourier basis as the
dictionary and equally spaced quadrature nodes are visually indistinguishable
from the eigenvalues of $\mathcal{K}$ when considered on $H^2(\bbA_r)^*$. As our next
example, we shall choose an expanding circle map which is less well-behaved,
resulting in spurious eigenvalues for a finite-size EDMD matrix. Let $S$ be the
circle map given by
\begin{equation}\label{eq:another_blaschke}\setlength\abovedisplayskip{6pt}\setlength\belowdisplayskip{6pt}
    S : \bbT \to \bbT,\quad z \mapsto \left(\frac{z - \mu}{1 - \bar{\mu} z}\right)^2, 
\end{equation}
which is uniformly expanding for $\mu\in \mathbb{D}$ with $|\mu| < \frac{1}{3}$.
The map extends analytically to a suitable open  annulus $\bbA_r$ containing
$\mathbb{T}$ such that the associated transfer operator $\mathcal{L}$ is compact
and its spectrum is given by 
$$\setlength\abovedisplayskip{6pt}\setlength\belowdisplayskip{6pt}
    \sigma \left( \left. \scrL \right|_{H^2(\bbA_r)} \right)
    = \left\{ S'(z^*)^n \mid n \in \bbN_0 \right\} \cup \left\{ \overline{S'(z^*)}^{\,n} \mid n \in \bbN_0 \right\} 
    \cup \{ 0 \},
$$
where $z^*\in \mathbb{D}$ is the unique attracting fixed point of $S$ in
$\mathbb{D}$, see \cite{Slipantschuk}. We use \cref{alg:resdmd} to compute
residuals using the Hardy--Hilbert norm and the Sobolev norm, see
\cref{fig:another_blaschke} (see also \cref{fig:blaschke_contour} for a
different depiction of the right panel of \cref{fig:another_blaschke}). In
both cases, we observe that the computed matrix $K$ has eigenvalues (orange
crosses) of large magnitude, which are not in the spectrum (black circles) of
$\mathcal{K}$ when considered on~$H^2(\bbA_r)^*$. The computed residuals are indeed
indicatively large in the region surrounding these eigenvalues, while being
small near the (accurately identified) first leading eigenvalues of
$\mathcal{K}$. Thus, the method indeed reliably distinguishes spurious
eigenvalues from the ``true'' eigenvalues of~$\mathcal{K}$.

\subsection{Alanine Dipeptide}\label{sec:molecule}

Alanine dipeptide is a standard nontrivial test system for studying conformation dynamics \cite{molecule2}. Conformations correspond to metastable subsets of the configuration space \cite{DDJS98}, and can be identified via eigenvectors associated with real eigenvalues near 1 of a discretized transfer operator \cite{attr}. In contrast, potential-energy-based methods often struggle due to the abundance of local minima \cite{molecule}. It is well known—see, e.g., \cite{RAMACHANDRAN196395}—that the dominant conformations of such molecules are largely governed by two backbone dihedral angles (see \cref{fig:molecule_skeleton}).

\begin{figure}[h]
    \centering
    \begin{subfigure}{0.45\textwidth}
        \centering
        \includegraphics[width=\textwidth]{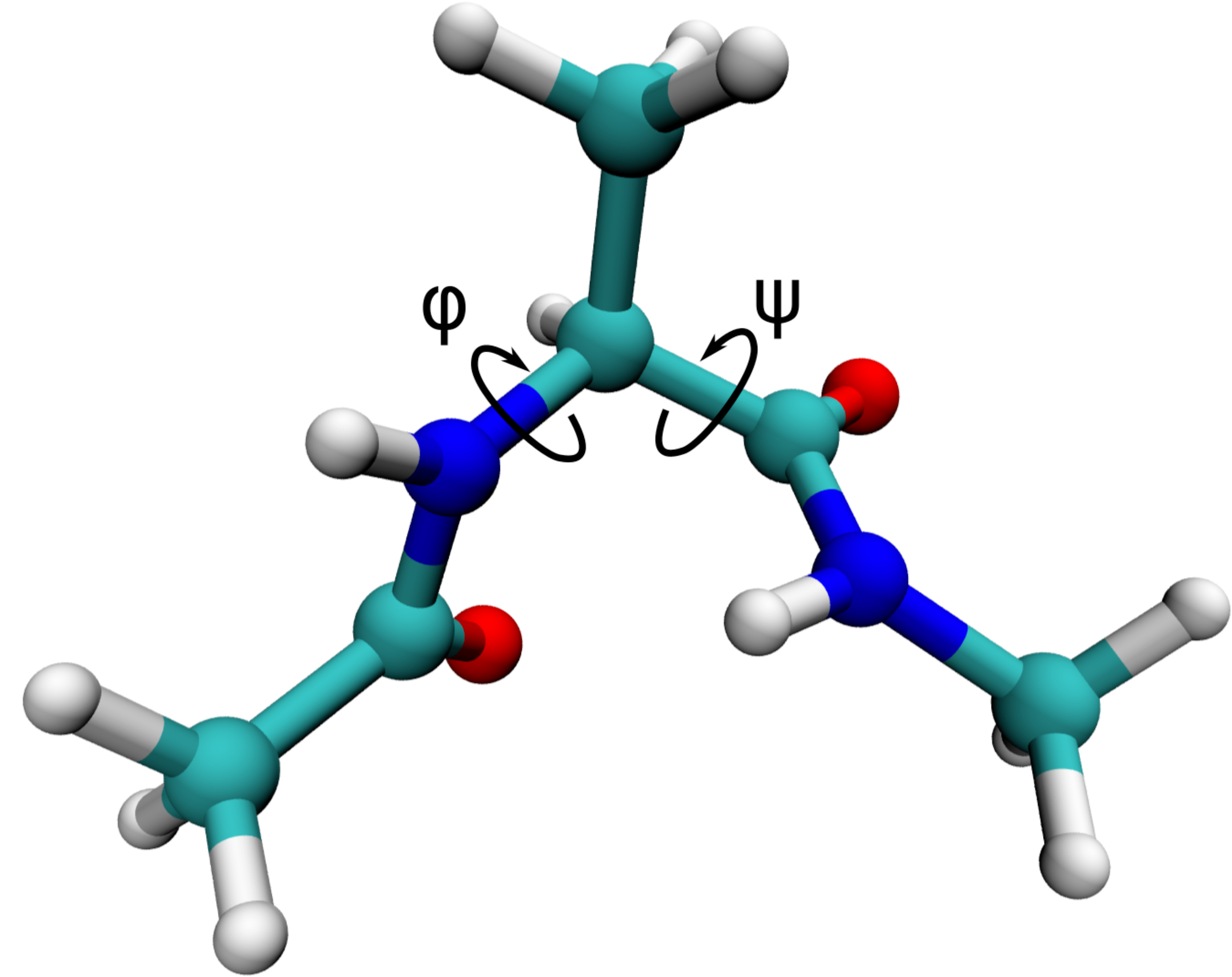}
    \end{subfigure}
    \caption{
        Alanine dipeptide molecule skeleton~\cite{skeleton}. The \emph{dihedral angles} $\varphi$ 
        and $\psi$ are the primary determining factors of the shape and chemical reaction 
        properties of the molecule. 
    }
    \label{fig:molecule_skeleton}
\end{figure}

Among many other approaches \cite{KlKoSch16}, recently, kernel-type methods have
been proposed for a discretization of the transfer
operator~\cite{molecule2,entropic}. Here, we additionally use
\cref{alg:kresdmd_correct} to verify the computed spectrum. We use trajectory
data of the heavy atoms gathered from experiments in \cite{molecule_experiment}.
After subsampling the trajectory data to use just every $50$th time step, we
obtain $M = 2500$ data points in $\bbR^{30}$. We apply
\cref{alg:kresdmd_correct} using the Gaussian kernel $k(w, z) =\exp ( - \| w - z
\|^2 / c^2 )$ with $c = 0.09$, the $2$-norm of the
empirical covariance matrix. The spectrum of $\smash{\widehat{ K }}$ is shown in
\cref{fig:molecule_spectrum}, together with the residuals. The result
indicates that the computed eigenvalues are indeed reliable. 

\begin{figure}[h]
    \centering
    \begin{subfigure}{0.45\textwidth}
        \centering
        \includegraphics[width=\textwidth]{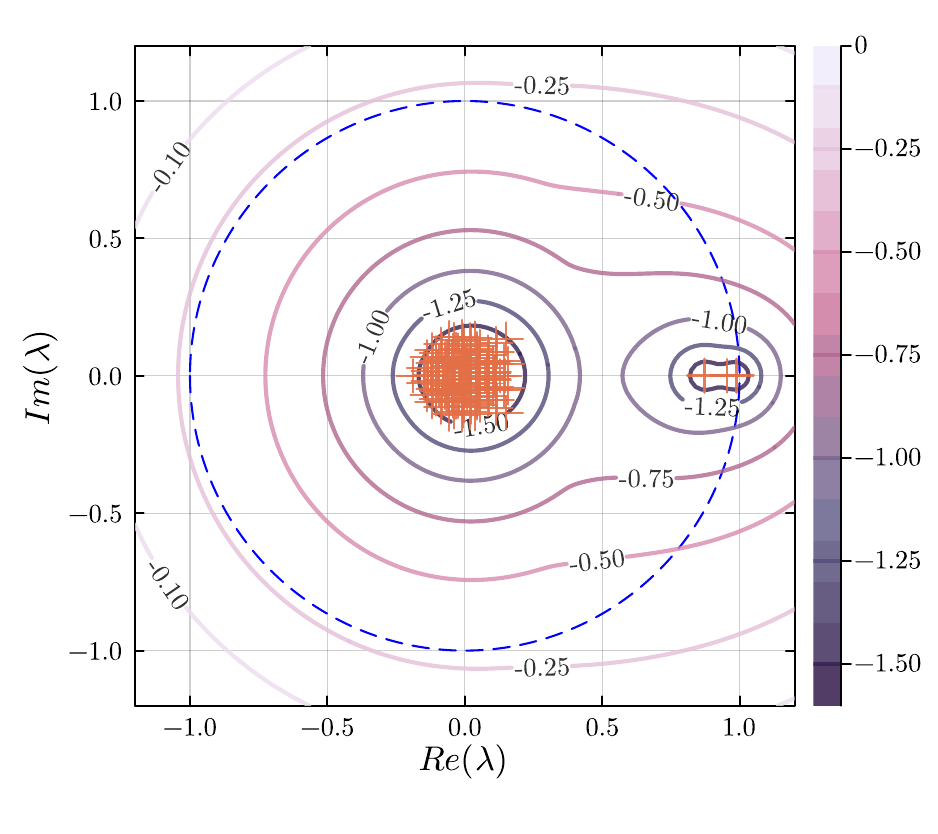}
    \end{subfigure}
    \caption{Spectrum of $\widehat{ K }$ with residuals computed by \cref{alg:kresdmd_correct}. 
        Contour lines are logarithmically scaled.}
    \label{fig:molecule_spectrum}
\end{figure}

For completeness, we also show the eigenfunctions at the two eigenvalues close
to 1 projected onto the two dihedral angles, cf. \cite{molecule2,entropic}, in
\cref{fig:molecule_eigs_kmeans}, left, which are used to detect almost
invariant/metastable sets via kmeans clustering
(\cref{fig:molecule_eigs_kmeans}, right). These figures agree with
previous findings \cite{molecule2}. All of the computations are done 
in the full $30$-dimensional space, and the observables use no a priori information on the dihedral angles.

\begin{figure}[h]
    \centering
    \begin{subfigure}{0.67\textwidth}
        \centering
        \includegraphics[width=\textwidth]{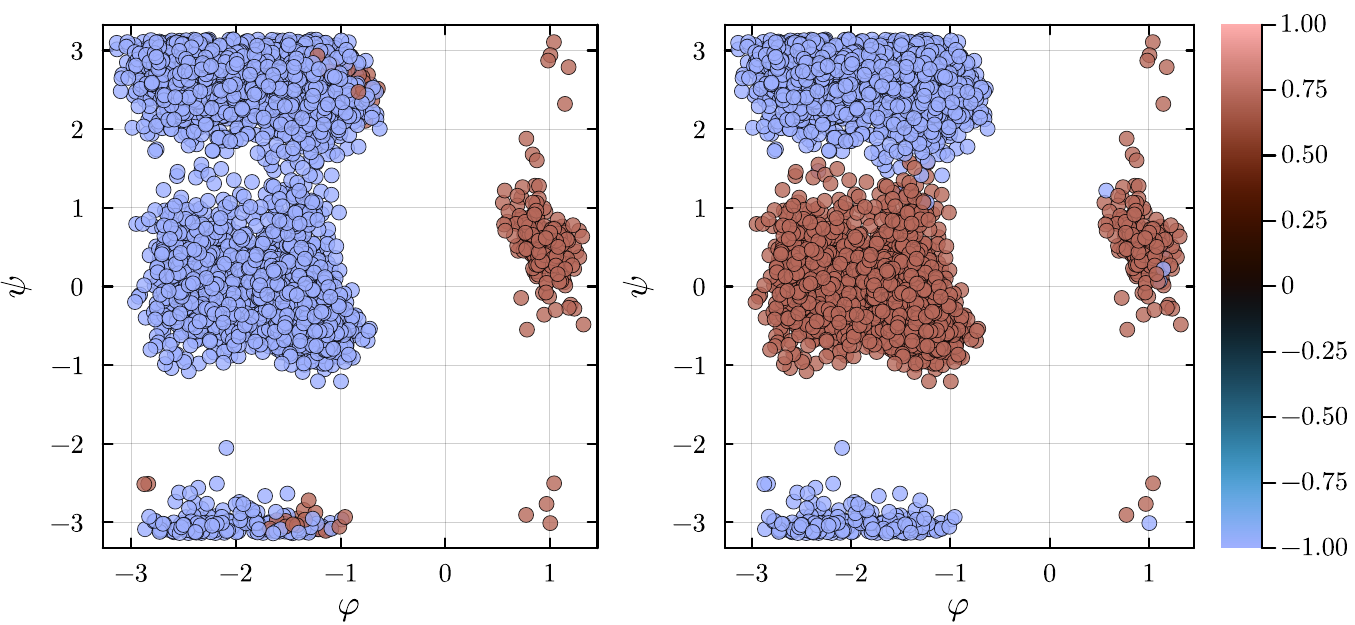}
    \end{subfigure}
    \hfill
    \begin{subfigure}{0.31\textwidth}
        \centering
        \includegraphics[width=\textwidth]{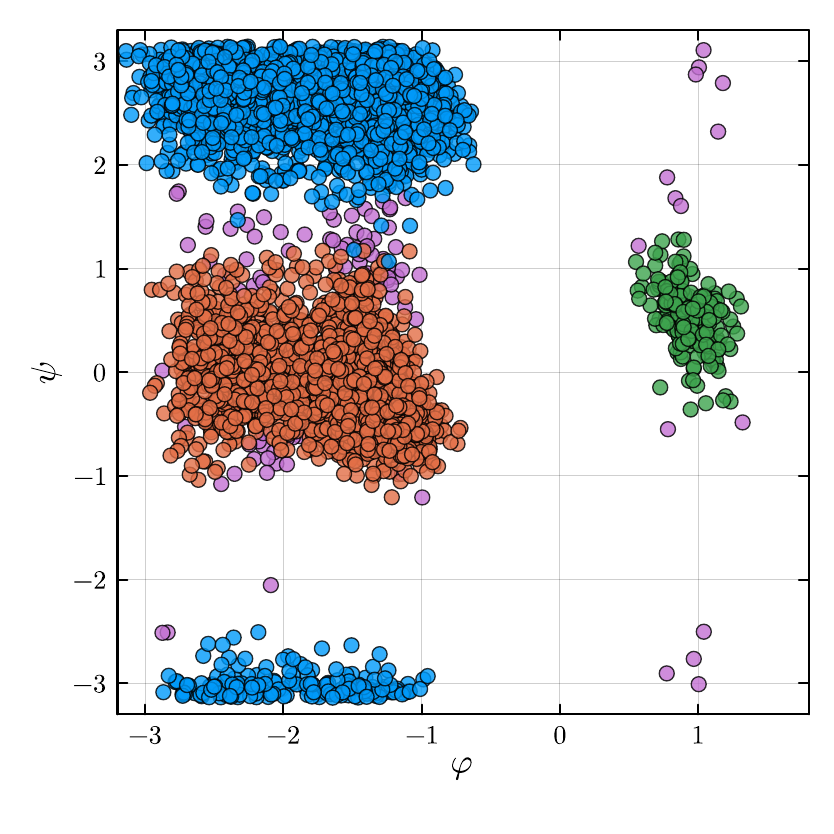}
    \end{subfigure}
    \caption{
        Left: First two nontrivial eigenfunctions of $\widehat{ K }$ for the alanine 
        dipeptide molecule, projected into the space of the two dihedral angles. 
        Right: k-means clustering of the eigenvectors, revealing the conformations, projected into the 
        space of the two dihedral angles.  
    }\label{fig:molecule_eigs_kmeans}
\end{figure}

\clearpage
\small
\bibliographystyle{siamplain}
\bibliography{main.bib}
\normalsize

\appendix

\section{Proof of \cref{lem:kres}}\label{sec:appendixA}

\begin{proof}[Proof of \cref{lem:kres}]
    We assume the snapshots are chosen as a deterministic quadrature scheme. 
    Note that the proof can be performed exactly the same with a random 
    quadrature, where \cref{eq:quadrature_M,eq:gram_M} are chosen to hold with 
    probability $\geq 1 - \delta$ for any fixed $0 < \delta < 1$. 
    We consider a finite approximation of the (in general) infinite Mercer
    dictionary. For $P \in \bbN$ define $\vY = [ \psi_1 \mid \ldots \mid
    \psi_P^{} ]$. Using this dictionary consider 
    $$\setlength\abovedisplayskip{6pt}\setlength\belowdisplayskip{6pt}
        \widecheck{ G } = \tfrac{1}{M} \vYX \vYX^* , \quad 
        \widecheck{ A } = \tfrac{1}{M} \vYY \vYX^* , \quad 
        \widecheck{ J } = \tfrac{1}{M} \vYY \vYY^* , 
    $$
    $$\setlength\abovedisplayskip{6pt}\setlength\belowdisplayskip{6pt}
        \widecheck{ U } ( = \widecheck{ U } (\lambda) ) = 
        \widecheck{ J } - \bar{\lambda} \widecheck{ A } - \lambda \widecheck{ A }^*
        + | \lambda |^2 \widecheck{ G } . 
    $$
    Notice that $\widecheck{ G }_{i j} = \sum_{\ell = 1}^P \mu_\ell^2 \phi_\ell
    (x_i) \overline{\phi_\ell} (x_j)$. Recall $\widehat{G}_{i j} = k ( x_j, x_i )$. 
    The analogous is true for $\widecheck{ A }$ and $\widecheck{ J }$. 

    Let $\epsilon > 0$. As we are concerned with small $\epsilon$ we can wlog
    assume $\epsilon < 1$. Choose $P = P (\epsilon)$ such that 
    \begin{equation}\setlength\abovedisplayskip{6pt}\setlength\belowdisplayskip{6pt}
        \label{eq:P_choice}
        \mu_{P + 1}^2 < \epsilon / 4 (1 + | \lambda |^2), \quad\quad 
        \max_{B \in \{ G, A, J \}}\ \left\| 
            \widetilde{ Q }^* (\widehat{ B } - \widecheck{ B }) \widetilde{ Q } 
        \right\|_{\bbC^M - op}^2 < \epsilon / 16 . 
    \end{equation}
    This is possible since $\mu_P \to 0$ and  $\sum_{\ell = 1}^P \mu_\ell^2
    \phi_\ell (z) \overline{\phi_\ell (w)} \to k(w, z)$ uniformly for all $w, z
    \in \Omega$ as $P \to \infty$ \cite{learning}. 

    Now choose $M = M (\epsilon, P, r)$ such that 
    \begin{equation}\setlength\abovedisplayskip{6pt}\setlength\belowdisplayskip{6pt}
        \label{eq:quadrature_M}
        \max_{\substack{
            \psi \in \spn \vY \\ 
            h \in \spn \tY \\ 
            \| \psi \|_{L^2}^{} = \| h \|_{L^2}^{} = 2
        }}\ 
        \left| 
            \frac{1}{M} \sum_{i = 1}^{M} \overline{\left[ (\scrK - \lambda I) \psi \right] (x_i)} \,h (x_i)
            - \left\langle\ (\scrK - \lambda I) \psi, h\ \right\rangle
        \right|^2
        \,<\, \frac{\epsilon}{4 P} . 
    \end{equation}
    This is possible since $\spn \tY$ and $\spn \vY$ are both
    finite-dimensional. \Cref{eq:quadrature_M} is a condition on the accuracy of the quadrature scheme induced by the data points $x_i$. Now, 
    $$\setlength\abovedisplayskip{6pt}\setlength\belowdisplayskip{6pt}
        \min_{\substack{h \in \spn \tY \\ \| h \|_{L^2}^{} = 1}}\ 
        \sum_{j = 1}^{\infty} \left| 
            \left\langle\ \psi_j, (\scrL - \lambda I) h\ \right\rangle 
        \right|^2
        = \min_{\substack{h \in \spn \tY \\ \| h \|_{L^2}^{} = 1}}\ 
        \sum_{j = 1}^{P} \left| 
            \left\langle\ \psi_j, (\scrL - \lambda I) h\ \right\rangle 
        \right|^2 
        \ + \ R_1
    $$
    where 
    \begin{multline*}\setlength\abovedisplayskip{6pt}\setlength\belowdisplayskip{6pt}
        | R_1 | \leq \max_{\substack{
            h \in \spn \tY \\ 
            \| h \|_{L^2}^{} = 1
        }}\ \sum_{j = P + 1}^{\infty} \left| \left\langle 
            \ \psi_j, (\scrL - \lambda I) h \ 
        \right\rangle \right|^2 \\ 
        \leq \mu_{P + 1}^2 \max_{\substack{
            h \in \spn \tY \\ 
            \| h \|_{L^2}^{} = 1
        }}\ \sum_{j = P + 1}^{\infty} \left| \left\langle 
            \ \phi_j, (\scrL - \lambda I) h \ 
        \right\rangle \right|^2 
        \leq \mu_{P + 1}^2 \left\| \scrL - \lambda I \right\|_{L^2 - op}^2
        < \epsilon / 4
    \end{multline*}
    where the final inequality is due to the triangle inequality since $\| \scrL
    \|_{L^2 - op} = 1$. Now, 
    $$\setlength\abovedisplayskip{6pt}\setlength\belowdisplayskip{6pt}
        \min_{\substack{
            h \in \spn \tY \\ 
            \| h \|_{L^2}^{} = 1
        }}\ 
        \sum_{j = 1}^{P} \left| 
            \left\langle\ \psi_j, (\scrL - \lambda I) h\ \right\rangle 
        \right|^2 
        = \min_{\substack{
            h \in \spn \tY \\ 
            \| h \|_{L^2}^{} = 1
        }}\ 
        \sum_{j = 1}^{P} \left| 
            \frac{1}{M} \sum_{i = 1}^{M} \overline{\left[ (\scrK - \bar{\lambda} I) \psi_j \right] (x_i)}\ h (x_i)
        \right|^2 + R_2
    $$
    where $| R_2 | < \epsilon / 4$ by \cref{eq:quadrature_M}. 
    Moreover, the above equation can be written as the minimum of the quadratic 
    function 
    \begin{equation}\setlength\abovedisplayskip{6pt}\setlength\belowdisplayskip{6pt}
        \label{eq:quadratic_form}    
        u \mapsto \sum_{j = 1}^{P} \left| 
            \frac{1}{M} \sum_{i = 1}^{M} \overline{\left[ (\scrK - \bar{\lambda} I) \psi_j \right] (x_i)}\ [ \tY u ] (x_i)
        \right|^2
    \end{equation}
    over the finite-dimensional space $\bbC^r$, constrained by another quadratic
    function, $u \mapsto u^* \tY^* \tY u$. The former also has a representation
    of the form $u \mapsto u^* \Xi u$ for some symmetric matrix $\Xi \in
    \bbC^{r \times r}$, and so can be solved by computing the smallest eigenvalue of
    $(\tY^* \tY)^{-1} \Xi$. Indeed, this property of quadratically-constrained
    quadratic minimization problems was already used in
    \cref{sec:residual_function} to compute $\res$. Hence, 
    \begin{multline*}\setlength\abovedisplayskip{6pt}\setlength\belowdisplayskip{6pt}
        \min_{\substack{h \in \spn \tY \\ \| h \|_{L^2}^{} = 1}}\ 
        \sum_{j = 1}^{P} \left| 
            \frac{1}{M} \sum_{i = 1}^{M} \overline{\left[ (\scrK - \bar{\lambda} I) \psi_j \right] (x_i)}\ h (x_i)
        \right|^2 \\ 
        = \min_{\substack{u \in \bbC^r \\ u^* \widetilde{ \Sigma }^2 u = 1}}\ 
        \sum_{j = 1}^{P} \left| 
            \frac{1}{M} \sum_{i = 1}^{M} \overline{\left[ (\scrK - \bar{\lambda} I) \psi_j \right] (x_i)}\ [ \tY u ] (x_i)
        \right|^2 + R_3
    \end{multline*}
    where 
    $$\setlength\abovedisplayskip{6pt}\setlength\belowdisplayskip{6pt}
        | R_3 | = \left| 
            \sigma_{\inf} ( (\tY^* \tY)^{-1} \Xi \,) 
            - \sigma_{\inf} ( \widetilde{ \Sigma }^{-2} \Xi \,)
        \right|
    $$
    and $\sigma_{\inf} (R)$ is the smallest eigenvalue of a symmetric 
    matrix $R$. Hence this is 
    $$\setlength\abovedisplayskip{6pt}\setlength\belowdisplayskip{6pt}
        | R_3 | \leq \left\| 
            (\tY^* \tY)^{-1} \Xi - \widetilde{ \Sigma }^{-2} \Xi 
        \right\|
        \leq \left\| (\tY^* \tY)^{-1} - \widetilde{ \Sigma }^{-2} \right\| 
        \left\| \Xi \right\| . 
    $$
    From the formulation of $\Xi$ in  \cref{eq:quadratic_form} it is clear that
    there is a $C > 0$ such that $\| \Xi \| < C$ for all $M$, since the
    quadrature scheme converges by assumption and $P < \infty$. Moreover, 
    by potentially increasing $M$, we can have 
    \begin{equation}\setlength\abovedisplayskip{6pt}\setlength\belowdisplayskip{6pt}
        \label{eq:gram_M}
        \left\| 
            (\tY^* \tY)^{-1} - \widetilde{ \Sigma }^{-2} 
        \right\|_{\bbC^r - op} < \epsilon / 4 C , 
    \end{equation}
    again since the quadrature scheme converges, and $\widetilde{ \Sigma }^2 =
    \widetilde{ \Sigma }^* \widetilde{ Q }^* \widetilde{ Q } \widetilde{ \Sigma
    } = \tYX^* \tYX$. We therefore have $| R_3 | < \epsilon / 4$. Now, 
    $$\setlength\abovedisplayskip{6pt}\setlength\belowdisplayskip{6pt}
        \min_{\substack{u^* \widetilde{ \Sigma }^2 u = 1}}\ 
        \sum_{j = 1}^{P} \left| 
            \frac{1}{M} \sum_{i = 1}^{M} \overline{\left[ (\scrK - \bar{\lambda} I) \psi_j \right] (x_i)}\ [ \tY u ] (x_i)
        \right|^2
        = \min_{\substack{u^* \widetilde{ \Sigma }^2 u = 1}}
        \frac{1}{M} \left\| (\vYY^* - \bar{\lambda} \vYX^*) \tYX u \right\|_{\bbC^M}^2
    $$
    which, by the substitution $w = \widetilde{ \Sigma } u$, 
    $$\setlength\abovedisplayskip{6pt}\setlength\belowdisplayskip{6pt}
        \min_{\substack{u^* \widetilde{ \Sigma }^2 u = 1}}
        \frac{1}{M} \left\| (\vYY^* - \bar{\lambda} \vYX^*) \tYX u \right\|_{\bbC^M}^2
        = \min_{w^* w = 1} w^* \widetilde{ Q }^* \left( 
            \widecheck{ J } 
            - \lambda \widecheck{ A } 
            - \bar{\lambda} \widecheck{ A }^* 
            + | \lambda |^2 \widecheck{ G }
        \right) \widetilde{ Q } w . 
    $$
    Finally, 
    $$\setlength\abovedisplayskip{6pt}\setlength\belowdisplayskip{6pt}
        \min_{w^* w = 1} w^* \widetilde{ Q }^* \left( 
            \widecheck{ J } 
            - \lambda \widecheck{ A } 
            - \bar{\lambda} \widecheck{ A }^* 
            + | \lambda |^2 \widecheck{ G }
        \right) \widetilde{ Q } w 
        = \min_{w^* w = 1} w^* \widetilde{ Q }^* \left( 
            \widehat{ J } 
            - \lambda \widehat{ A } 
            - \bar{\lambda} \widehat{ A }^* 
            + | \lambda |^2 \widehat{ G }
        \right) \widetilde{ Q } w + R_4
    $$
    where $| R_4 | < \epsilon / 4$ by \cref{eq:P_choice} and the same 
    argumentation as for $R_3$. This is precisely 
    $$\setlength\abovedisplayskip{6pt}\setlength\belowdisplayskip{6pt}
        \min_{w^* w = 1} w^* \widetilde{ Q }^* \left( 
            \widehat{ J } 
            - \lambda \widehat{ A } 
            - \bar{\lambda} \widehat{ A }^* 
            + | \lambda |^2 \widehat{ G }
        \right) \widetilde{ Q } w
        = \kres (\lambda; M, r)^2 . 
    $$
    Combining the four error terms now yields the claim. 
\end{proof}

\begin{remark}
    The proof of Lemma \ref{lem:kres} provides explicit error bounds which can
    be used for validated numerics, provided one has control over the quadrature
    error. For this, higher-order or analytically studied quadrature schemes can be used.
    Furthermore, the limit $P \to \infty$ is merely a tool in the proof, and
    does not need to be controlled explicitly, only a choice for $M$ and $r$
    must be made. In this way, the algorithm is still optimal in the sense of
    solvability complexity index; see \cite{benartzi2020}.
\end{remark}

\end{document}